\documentclass[10pt,reqno]{amsart}

\usepackage{amsthm, mathrsfs,amssymb,amsmath}
\usepackage{enumerate}
\usepackage[hidelinks]{hyperref}
\usepackage{xcolor}
\hypersetup{
	colorlinks,
	linkcolor={red!50!black},
	citecolor={green!50!black},
	urlcolor={red!80!black}
}
\makeatletter
\@namedef{subjclassname@2010}{%
	\textup{2010} Mathematics Subject Classification}
\makeatother

\allowdisplaybreaks

\newtheorem{thm}{Theorem}[section]
\newtheorem{lem}[thm]{Lemma}

\newtheorem{cor}[thm]{Corollary}

\theoremstyle{definition}

\theoremstyle{remark}

\title[Titchmarsh Theorems, Paley inequality and $L^p-L^q$ boundedness ]{Titchmarsh Theorems, Hausdorff-Young-Paley inequality and $L^p-L^q$ boundedness of Fourier multipliers  on Harmonic $NA$ groups}

\author{Vishvesh Kumar}
\address{Vishvesh Kumar \endgraf
	Department of Mathematics: Analysis, Logic and Discrete Mathematics
	\endgraf
	Ghent University, Belgium}
\endgraf
\email{vishveshmishra@gmail.com, Vishvesh.Kumar@UGent.be}
\author[Michael Ruzhansky]{Michael Ruzhansky}
\address{
	Michael Ruzhansky
	\endgraf
	Department of Mathematics: Analysis, Logic and Discrete Mathematics
	\endgraf
	Ghent University, Belgium
	\endgraf
	and
	\endgraf
	School of Mathematical Sciences
	\endgraf
	Queen Mary University of London
	\endgraf
	United Kingdom
	\endgraf
	{\it E-mail address} {\rm michael.ruzhansky@ugent.be}
}

\begin{document}
	
	\begin{abstract} In this paper we extend   classical Titchmarsh theorems on the Fourier transform of  H$\ddot{\text{o}}$lder-Lipschitz functions to the setting of harmonic $NA$ groups,  which relate smoothness properties of functions to the growth and integrability of their Fourier transform. We prove a Fourier multiplier theorem for $L^2$-H$\ddot{\text{o}}$lder-Lipschitz spaces on Harmonic $NA$ groups. We also derive conditions and a characterisation of Dini-Lipschitz classes on Harmonic $NA$ groups in terms of the behaviour of their Fourier transform. Then, we  shift our attention to the  spherical analysis on Harmonic $NA$ group. Since the spherical analysis on these groups fits well in the setting of Jacobi analysis we prefer to work in the Jacobi setting.  We prove $L^p$-$L^q$ boundedness of Fourier multipliers by extending a classical theorem of H$\ddot{\text{o}}$rmander to the Jacobi analysis setting. On the way to accomplish this classical result we prove Paley-type inequality and Hausdorff-Young-Paley inequality. We also establish $L^p$-$L^q$ boundedness of spectral multipliers of the Jacobi Laplacian.

	\end{abstract}
	\maketitle
	\tableofcontents 
	
	\section{Introduction}  \label{Sec1}
	\noindent Harmonic $NA$ groups (also known as Damek-Ricci spaces) form a class of solvable Lie groups (non-unimodular), equipped with a left invariant Riemannian metric, called
“harmonic extension of $H$-type groups”, is a natural generalisation of the Iwasawa $NA$
groups of the real rank-one simple Lie groups. In particular, as  Riemannian manifolds,  the class of harmonic $NA$ groups contains rank-one symmetric spaces of non-compact type as a very small subclass. Although, the rank one noncompact Riemannian symmetric spaces are the most  discernible prototypes of harmonic $NA$ groups, they have many contrasts  with these. From the geometric point of view, Harmonic $NA$ groups, in general, are not symmetric which can be seen through the fact that the geodesic inversion is not an isometry. On the other hand, these spaces cannot have Kunze-Stein property for being noncompact amenable groups. 

The studies of relations between the smoothness of the functions and the growth and (or) the integrability of their Fourier coefficient are  among the classical and fundamental problems in Fourier analysis. These types of studies include the classical theorems of Fourier transform like Riemann-Lebesgue Lemma relating the integrability of a functions and decay of its Fourier transform, and the Hausdorff-Young inequality relating the integrability of a function and its Fourier transform.  Indeed, the Hausdorff-Young inequality on $\mathbb{R}$ states that for $f \in L^p(\mathbb{R}),\, 1 \leq p \leq 2,$ its Fourier transform $\widehat{f} \in L^{p'}(\mathbb{R})$ with $\frac{1}{p}+\frac{1}{p'}=1.$ Titchmarsh \cite{Titch} studies the problem of how much this fact can be strengthened if $f$ additionally satisfies a Lipschitz condition. Indeed, in this vain, Titchmarsh \cite{Titch} showed the this can be significantly improved. The first Titchmarsh theorem we deal with \cite[Theorem 37]{Titch} in this paper is recalled below: We define H$\ddot{\text{o}}$lder-Lipschitz space $\textnormal{Lip}_{\mathbb{R}}(\alpha, p)$ for $0< \alpha \leq 1 $ and $1<p<\infty$ by 
$$\textnormal{Lip}_{\mathbb{R}}(\alpha, p):= \left\{f \in L^p(\mathbb{R}): \|f(t+\cdot)-f(\cdot)\|_{L^p(\mathbb{R})}= O(t^\alpha)\,\,\,\text{as}\,\, t \rightarrow 0\right\}.$$

{\bf Theorem A.} Let $0<\alpha \leq 1$ and $1 <p \leq 2.$ If $f \in \textnormal{Lip}_{\mathbb{R}}(\alpha, p)$, then $\widehat{f} \in L^\beta(\mathbb{R})$ for $$\frac{p}{p+\alpha p-1}<\beta \leq p',\,\,\,\,\,\,\, \frac{1}{p}+\frac{1}{p'}=1.$$

To discuss the improvement in the conclusion $\widehat{f} \in L^{\beta}(\mathbb{R})$ in Theorem A above in comparison to the Hausdorff-Young inequality, for instance we can take $\alpha=\frac{1}{2}$ and $p=2$ then $\frac{p}{p+\alpha p-1}=1$, therefore $\widehat{f} \in L^\beta(\mathbb{R})$ for all $\beta \in (1, 2].$  The above theorem has been investigated in several different settings, e.g. on compact symmetric spaces of rank one, noncompact symmetric spaces of rank one, compact homogeneous manifolds, and the Euclidean space $\mathbb{R}^n$  \cite{Platonov, Platonov1, DDR, Bray, Younis1, Younis2}. Notably, our motivation is the work of Bray \cite{Bray} on $\mathbb{R}^n$ where he extends Theorem A in higher dimenstions, and a recent work of the third author \cite{DDR} on compact homogeneous spaces, as well as the work of Ray and Sarkar \cite{RS} where they proved a version of the Hausdorff-Young inequality for Damek-Ricci spaces.  In the context of Dini-Lipschitz spaces, Theorem A has been studied in the setting on $\mathbb{R}$ by Younis \cite{Younis}. In Theorem A' and Theorem \ref{DTitch2} we extend the Theorem A to the setting of Damek-Ricci spaces. It is worth noting that Theorem \ref{DTitch2} is already new in rank one noncompact symmetric space setting. 

The second Titchmarsh theorem \cite[Theorem 85]{Titch} which is also of our interest is stated below.

{\bf Theorem B.} Let $0<\alpha \leq 1$ and $f \in L^2(\mathbb{R}).$ Then $f \in \textnormal{Lip}_{\mathbb{R}}(\alpha, 2)$ if and only if 
$$\int_{|\xi| >\frac{1}{t}} |\widehat{f}(\xi)|^2\, d\xi = O(t^{2\alpha})\,\,\,\,\text{as}\,\, t \rightarrow 0.$$
This theorem was extended  by Bray \cite{Bray} to  higher dimensional Euclidean spaces in a more general setting using multipliers by modifying the technique given in the seminal paper of Platonov \cite{Platonov} in the case of rank one noncompact symmetric spaces. For an overview of extensions of this theorem in different settings we refer to \cite{DDR, Younis1, Younis2}. In Theorem B' and Theorem \ref{Them4.1} we extend this theorem to the  setting of Helgason Fourier transform on Damek-Ricci spaces. 

To state our main results briefly, we first need to define the H$\ddot{\text{o}}$lder-Lipschitz spaces $\textnormal{Lip}_{S}(\alpha, p)$ in a harmonic $NA$ group $S$.  We define $\textnormal{Lip}_{S}(\alpha, p)$ using the spherical mean operator $M_t,\, t\in \mathbb{R}_+,$  on Damek-Ricci spaces (see Section \ref{Ess} for definition and more details). The use of the  spherical mean operator is inspired by the work Platonov \cite{Platonov, Platonov1} and Bray \cite{Bray}. In the setting of Damek-Ricci spaces, the spherical mean operator was used in \cite{KRS} to answer some questions initially posed by Bray and Pinsky \cite{BrayPinsky} regarding the growth of Fourier transform. Thus, the H$\ddot{\text{o}}$lder-Lipschitz space $\textnormal{Lip}_{S}(\alpha, p)$ is defined as
  $$\textnormal{Lip}_{S}(\alpha, p):= \{f \in L^p(S): \|M_tf-f\|_p= O(t^\alpha)\,\,\,\,\, \text{as}\,\, t \rightarrow 0\}.$$
  Now, we state both of our results which are the suitable analogues of Titchmarsh theorems in the context of harmonic $NA$ groups. We denote by $\widetilde{f}$ the (Helgason) Fourier transform of $f$ on the  harmonic $NA$ group $S$ (see \cite{ACB97}).  Also, we set  $\gamma_p= \frac{2}{p}-1$ and so $\gamma_{p'}=\frac{2}{p'}-1=-\gamma_p,$ where $p'$ is the Lebesgue conjugate of $p.$ 
  
  {\bf Theorem A'.} \label{A'}  Let $S$ be a harmonic NA group of dimension $d$. Let $0 <\alpha \leq 1,\, 1<p \leq 2,$ and let $p'$ be  such that $\frac{1}{p}+\frac{1}{p'}=1.$
    Let $f \in \text{Lip}_S(\alpha; p).$ Then
    the function $F$ defined by 
$$F(\lambda):=   \left(\int_N |\widetilde{f}(\lambda+i \gamma_{p'} \rho, n)|^{p'} \,dn \right)^{\frac{1}{p'}},\,\, 1<p \leq 2,
$$
    belongs to $ L^\beta((0, \infty), |c(\lambda)|^{-2}\, d \lambda)$ provided that 
    \begin{align*} \label{beta}
        \frac{dp}{dp+\alpha p-d} < \beta \leq p'.
    \end{align*}

{\bf Theorem B'.}  Let $0<\alpha \leq 1$ and $f \in L^2(S).$ Then $f \in \text{Lip}_S(\alpha; 2)$ if and only if 
	$$ \int_{\lambda >\frac{1}{t}} \int_N |\widetilde{f}(\lambda, n)|^2 dn |c(\lambda)|^{-2} \, d \lambda = O(t^{2 \alpha})\,\,\,\,\text{as}\,\, t \rightarrow 0.$$
	
	We also prove the aforementioned theorem for Dini-Lipschitz spaces on harmonic $NA$ groups. As an application of the above characterisation we show a Fourier multiplier theorem for H$\ddot{\text{o}}$lder-Lipschitz spaces on Damek-Ricci spaces.

	\begin{cor} \label{Coro1}
	 Let $0 \leq \gamma <1$ and let $h$ be an even bounded  measurable function on $\mathbb{R}$ such that 
	 $$|h(\lambda)|\leq C \langle \lambda \rangle^{-\gamma},$$ where $\langle \lambda \rangle = (\lambda^2+\frac{Q^2}{2})^{\frac{1}{2}},$ and $Q$ denotes the homogeneous dimension of $N.$ 
	 Let $A$ be the Fourier multiplier with symbol $h$, i.e., given by $\widetilde{Af}(\lambda, n)= h(\lambda) \widetilde{f}(\lambda, n)$ for all $\lambda \in \mathbb{R}_+$ and $n \in N.$ Then 
	 $$A:\text{Lip}_S(\alpha; 2) \rightarrow \text{Lip}_S(\alpha+\gamma;2)$$ is bounded for all $\alpha$ such that $0<\alpha<1-\gamma.$
	\end{cor}
The multiplier theorem in Corollary \ref{Coro1} complements the other known multipliers theorems on Damek-Ricci spaces \cite{Anker96}. 

Now, we would like to divert our attention to the second part of the paper where we establish an $L^p$-$L^q$  multiplier theorem for the Jacobi transform. In particular these results are true for the radial multipliers on Damek-Ricci spaces as the spherical analysis on Damek-Ricci spaces fits well in the setting of Jacobi analysis as observed in \cite{Anker96}. Here we deal with $L^p$-$L^q$ multipliers as opposed to the $L^p$-multipliers for which theorems of Mihlin-H$\ddot{\text{o}}$rmander or Marcinkiewicz type provide results for Fourier multiplier in different settings based on the regularity of the symbol. We cite here \cite{Hormander, Hormander1960,Anker90, Anker92, Anker96, Cow3, Cow1, Cow2, Cow4, Ruzwirth,Del, John, Astengo2, Astengo, MPR, Bloom} to mention a few of them.  In \cite{Anker96}, several radial $L^p$-multipliers theorems have been established including the Mihlin-H$\ddot{\text{o}}$rmander multiplier theorem on Damek-Ricci spaces.    The Paley-type inequality describes the growth of the Fourier transform of a function in terms of its $L^p$-norm. Interpolating the Paley-inequality with the Hausdorff-Young inequality one can obtain the following H\"ormander's version of the  Hausdorff-Young-Paley inequality,
\begin{equation}\label{3}
    \left(\int\limits_{\mathbb{R}^n}|(\mathscr{F}f)(\xi)\phi(\xi)^{ \frac{1}{r}-\frac{1}{p'} }|^r\right)^{\frac{1}{r}}\leq \Vert f \Vert_{L^p(\mathbb{R}^n)},\,\,\,1<p\leq r\leq p'<\infty, \,\,1<p<2.
\end{equation} Also, as a consequence  of the Hausdorff-Young-Paley inequality, H\"ormander \cite[page 106]{Hormander1960} proves that the condition 
\begin{equation}\label{4}
    \sup_{t>0}t^b\{\xi\in \mathbb{R}^n:m(\xi)\geq t\}<\infty,\quad \frac{1}{p}-\frac{1}{q}=\frac{1}{b},
\end{equation}where $1<p\leq 2\leq q<\infty,$ implies the existence of a bounded extension of a Fourier multiplier $T_{m}$ with symbol $m$ from $L^p(\mathbb{R}^n)$ to $L^q(\mathbb{R}^n).$ Recently, the third author with his collaborator R. Akylzhanov  studied H\"ormander classical results for unimodular locally compact groups and homogeneous spaces \cite{ARN, AR}. In \cite{AR}, the key idea behind the extension of H\"ormander theorem is the reformulation of this theorem as follows: 
$$\|T_m\|_{L^p(\mathbb{R}^n) \rightarrow L^q(\mathbb{R}^n)} \lesssim \sup_{s>0} s\left( \int_{ \{\xi \in \mathbb{R}^n :\, m(\xi) \geq s\} } d\xi \right)^{\frac{1}{p}-\frac{1}{q}} \simeq \|m\|_{L^{r, \infty}(\mathbb{R}^n)} \simeq \|T_m\|_{L^{r, \infty}(\textnormal{VN}(\mathbb{R}^n))},$$ where $\frac{1}{r}=\frac{1}{p}-\frac{1}{q},$ $\|m\|_{L^{r, \infty}(\mathbb{R}^n)}$ is the Lorentz norm of $m,$ and $\|T_m\|_{L^{r, \infty}(\textnormal{VN}(\mathbb{R}^n))}$ is the norm of the operator $T_m$ in the Lorentz space on the group von Neumann algebra $\textnormal{VN}(\mathbb{R}^n)$ of $\mathbb{R}^n.$ They use the Lorentz spaces and group von Neumann algebra techniques for extending it to general locally compact unimodular groups. The unimodularity assumption has its own advantages such as existence of the canonical trace on the group von Neumann algebra and consequently,  Plancherel formula and the Hausdorff-Young inequality. They also pointed out that the unimodularity can be avoided by using the Tomita-Takesaki modular theory and the Haagerup reduction technique. In our case, Harmonic $NA$ groups are nonunimodular groups but we are dealing with the Helgason Fourier transform and Harish Chandra transform instead of the group Fourier transform. So, the groups von Neumann algebra techniques are more intricate to implement. 
Untill now, there does not exist a suitable and optimal version of the  Hausdorff-Young inequality for the (Helgason) Fourier transform for Damek-Ricci space, in particular, for rank one noncompact symmetric spaces, although several attempts has been made for it and consequently, different version of the Hausdorff-Young inequality were proved, for example, see \cite{RS, MRSS}. On the other hand, spherical analysis on Damek-Ricci spaces (\cite{Anker96})  fits perfectly in the  well-established Jacobi analysis setup \cite{Koorn, FK,FK2}. Therefore, we choose to work in the setting of Jacobi analysis, in particular, the spherical analysis on Damek-Ricci spaces. 
In this paper, we prove Paley-type inequality, Hausdorff-Young-Paley inequality and H\"ormander multiplier theorem for Jacobi transform on half line $\mathbb{R}_+$. For rest of the section, we assume that $\alpha \neq -1, -2, \ldots$ and $\alpha \geq \beta > \frac{-1}{2}.$ We set $A_{\alpha, \beta}(t)= (2 \sinh{t})^{2\alpha+1} (2 \cosh{t})^{2\beta+1},\,\,t>0$ and denote by $c(\lambda),$ a multiple of the  meromorphic Harish-Chandra function given by the formula
	$$c(\lambda):=\frac{2^{\rho-i\lambda} \Gamma(\alpha+1) \Gamma (i\lambda)}{ \Gamma(\frac{1}{2}(\rho+i\lambda)) \Gamma(\frac{1}{2}(\rho+i \lambda)-\beta)}.$$ We refer to Section \ref{Jacobinota} for more details and all the notation used here.
The following theorem is an analogue of the  Paley inequality. 

\begin{thm}[\bf Paley-type inequality]
	      Suppose that $\psi$ is a positive function  on $\mathbb{R}_+$  satisfying the condition 
	    \begin{equation}
	        M_\psi := \sup_{t>0} t \int_{\underset{\psi(\lambda)>t}{\lambda \in \mathbb{R}_+}} |c(\lambda)|^{-2}\, d\lambda <\infty.
	        \end{equation}
	        Then for  $f \in L^p(\mathbb{R}_+, A_{\alpha, \beta}(t) dt),$ $1<p\leq 2,$ we have 
	        \begin{align} \label{Paleyin}
	            \left( \int_0^\infty |\widehat{f}(\lambda)|^p\, \psi(\lambda)^{2-p} |c(\lambda)|^{-2}\, d\lambda \right)^{\frac{1}{p}} \lesssim M_{\psi}^{\frac{2-p}{p}}\, \|f\|_{L^p(\mathbb{R}_+, A_{\alpha, \beta}(t) dt)}.
	        \end{align}
	\end{thm}
By interpolating the Hausdorff-Young inequality and Paley-type inequality we get the following Hausdorff-Young-Paley inequality for the Jacobi transform.
\begin{thm}[\bf Hausdorff-Young-Paley inequality]  Let $1<p\leq 2,$ and let   $1<p \leq b \leq p' \leq \infty,$ where $p'= \frac{p}{p-1}.$ If $\psi(\lambda)$ is a positive function on $\mathbb{R}_+$ such that 
 \begin{equation}
	        M_\psi := \sup_{t>0} t \int_{\underset{\psi(\lambda)>t}{\lambda \in \mathbb{R}_+}} |c(\lambda)|^{-2}\, d\lambda
	        \end{equation}
is finite, then for every $f \in L^p(\mathbb{R}_+, A_{\alpha, \beta}(t) dt)$ 
 we have
\begin{equation} 
    \left( \int_{\mathbb{R}_+}  \left( |\widehat{f}(\lambda)| \psi(\lambda)^{\frac{1}{b}-\frac{1}{p'}} \right)^b |c(\lambda)|^{-2}\, d\lambda  \right)^{\frac{1}{b}} \lesssim M_\varphi^{\frac{1}{b}-\frac{1}{p'}} \|f\|_{L^p(\mathbb{R}_+, A_{\alpha, \beta}(t) dt)}.
\end{equation}
\end{thm}
Finally we establish the following $L^p$-$L^q$ boundedness result for multipliers of Jacobi transform.
\begin{thm} \label{Jacobimultin}  Let $1<p \leq 2 \leq q<\infty$. Suppose that $T$ is a Jacobi-Fourier multiplier with symbol $h,$ that is, $$\widehat{Tf}(\lambda)= h(\lambda) \widehat{f}(\lambda),\,\,\,\lambda \in \mathbb{R}_+ ,$$
 where $h$ is an bounded measurable even function on $\mathbb{R}.$  Then we have 
$$\|T\|_{L^p(\mathbb{R}_+, A_{\alpha, \beta}(t) dt) \rightarrow L^q(\mathbb{R}_+, A_{\alpha, \beta}(t) dt)}\lesssim \sup_{s>0} s \left[ \int_{\{ \lambda \in \mathbb{R}_+: |h(\lambda)|>s\}} |c(\lambda)|^{-2}\, d\lambda \right]^{\frac{1}{p}-\frac{1}{q}}.$$
   \end{thm}
   
   Now, we apply Theorem \ref{Jacobimultin} to prove the $L^p$-$L^q$ boundedness of spectral multipliers for Jacobi Laplacian $L:=-\mathcal{L}_{\alpha, \beta}=- \frac{d^2}{dt^2}- ((2\alpha+1)\coth{t}(2\beta+1)\tanh{t}) \frac{d}{dt}.$ Set $\rho=\alpha+\beta+1.$
   
   \begin{thm} 
Let $1<p \leq 2 \leq q <\infty$ and let $\varphi$ be a monotonically  decreasing continuous function on $[\rho^2, \infty)$ such that $\lim_{u \rightarrow \infty}\varphi(u)=0.$ Then we have 
\begin{equation*}
    \|\varphi(L)\|_{\textnormal{op}} \lesssim \sup_{u>\rho^2} \varphi(u)  \begin{cases} (u-\rho^2)^{(\frac{1}{p}-\frac{1}{q})} & \quad \textnormal{if} \quad (u-\rho^2)^{\frac{1}{2}} \leq 1, \\ (u-\rho^2)^{(\alpha+1)(\frac{1}{p}-\frac{1}{q})} & \quad \textnormal{if} \quad (u-\rho^2)^{\frac{1}{2}}>1,  \end{cases}
\end{equation*} where  $\|\cdot\|_{\textnormal{op}} $ denotes the operator norm from $L^p(\mathbb{R}_+, A_{\alpha, \beta}(t) dt)$ to $L^q(\mathbb{R}_+, A_{\alpha, \beta}(t) dt).$ 
\end{thm}

	\section{ Essentials about harmonic $NA$ groups} \label{Ess}
	
	For basics of harmonic $NA$ groups and Fourier analysis on them, one can refer to papers \cite{Damek,DamekRicci,DamekRicci1, Di-Blasio,Anker96,ACB97,CDK91, RS, KRS, Kaplan, Astengo, Astengo2}. However, we recall necessary definitions, notation and  terminology that we shall use in this paper.
	
	 Let $\mathfrak{n}$ be a two-step nilpotent Lie algebra, equipped with an inner product $\langle \,,\, \rangle$ . Denote by $\mathfrak{z}$ the center of $\mathfrak{n}$ and by $\mathfrak{v}$ the orthogonal complement of $\mathfrak{z}$ in $\mathfrak{n}$ with respect to the inner product of $\mathfrak{n}.$  
	We assume that dimensions of $\mathfrak{v}$ and $\mathfrak{z}$ are $m$ and $l$ respectively as real vector spaces. The Lie algebra $\mathfrak{n}$ is $H$-type algebra if for every $Z \in \mathfrak{z},$ the map $J_Z:\mathfrak{v} \rightarrow \mathfrak{v}$  defined by 
	$$\langle J_Z X, Y\rangle = \langle Z, [X, Y] \rangle,\,\,\,\,\,\,X,Y \in \mathfrak{v},\, Z\in \mathfrak{z},$$ satisfies the condition $J_Z^2=-\|Z\|^2I_{\mathfrak{v}},$ where $I_{\mathfrak{v}}$ is the identity operator on $\mathfrak{v}.$ In \cite{Kaplan}, Kaplan proved that for $Z \in \mathfrak{z}$ with $\|Z\|=1$ one has $J_Z^2=-I_{\mathfrak{v}}$; that is, $J_Z$ induced a complex structure on $\mathfrak{v}$ and hence $m=\dim(\mathfrak{v})$ is always even.  A connected and simply connected Lie group $N$ is called $H$-type if its Lie algebra is of $H$-type. The exponential map is a diffeomorphism as $N$ is nilpotent, we can parametrize the element of $N=\exp{\mathfrak{n}}$ by $(X, Z)$, for $X \in \mathfrak{v}$  and $Z \in \mathfrak{z}.$ The multiplication on $N$ follows from the Campbell-Baker-Hausdorff formula given by 
	$$(X, Z) (X',Z')= (X+X', Z+Z'+\frac{1}{2}[X,X']).$$
	The group $A= \mathbb{R}_+^*$ acts on $N$ by nonisotropic dilations as follows: $(X, Y) \mapsto (a^{\frac{1}{2}}X, aZ).$ Let $S=N \ltimes A$ be the semidirect product of $N$ with $A$ under the aforementioned action. The group multiplication on $S$ is defined by 
	$$(X,Z,a)(X',Z',a')= (X+a^{\frac{1}{2}}X', Z+aZ'+\frac{1}{2}a^{\frac{1}{2}} [X, X'], aa').$$
	Then $S$ is a solvable (connected and simply connected) Lie group with Lie algebra $\mathfrak{s}=\mathfrak{z}\oplus \mathfrak{v} \oplus \mathbb{R}$ and Lie bracket 
	$$[(X, Z, \ell), (X', Z', \ell')]= (\frac{1}{2} \ell X'-\frac{1}{2} \ell'X, \ell Z'-\ell'Z+[X, X]', 0).$$
	The group $S$ is equipped with the left-invariant Riemannian metric induced by 
	$$ \langle (X, Z, \ell), (X', Z', \ell')\rangle= \langle X, X'\rangle+\langle Z, Z'\rangle+\ell \ell'$$ on $\mathfrak{s}.$ The homogneous dimension of $N$ is equal to $\frac{m}{2}+l$ and will be denoted by $Q.$ At times, we also use symbol $\rho$ for $\frac{Q}{2}.$ Hence $\dim(\mathfrak{s})=m+l+1,$ denoted by $d.$ The associated left Haar measure $dx$ on $S$ is given by $a^{-Q-1} dX dZ da,$ where $dX,\, dZ$ and $da$ are the Lebesgue measures on $\mathfrak{v}, \mathfrak{z}$ and $\mathbb{R}_+^*$ respectively. The element of $A$ will be identified with $a_t=e^t,$ $t \in \mathbb{R}.$ The group $S$ can be realized as the unit ball $B(\mathfrak{s})$ in $\mathfrak{s}$ using the Cayley transform $C: S \rightarrow  B(\mathfrak{s})$ (see \cite{Anker96}).
	
	To define (Helgason) Fourier transform on $S$ we need to introduce the notion of the Poisson kernel (\cite{ACB97}). The Poisson Kernel $\mathcal{P}:S \times N \rightarrow \mathbb{R}$ is defined by $\mathcal{P}(na_t, n')= P_{a_t}(n'^{-1}n),$ where $$ P_{a_t}(n)= P_{a_t}(X, Z)=C a_t^Q \left( \left(a_t+\frac{|X|^2}{4} \right)^2+|Z|^2 \right)^{-Q},\,\,\,\, n=(X, Z) \in N.$$
	The value of $C$ is suitably adjusted so that $\int_N P_a(n) dn=1$ and $P_1(n) \leq 1.$ The Poisson kernel satisfies several useful properties (see \cite{KRS,RS,ACB97}), we list here a few of them. For $\lambda \in \mathbb{C},$ the complex power of the Poisson kernel is defined as 
	$$\mathcal{P}_\lambda(x, n)= \mathcal{P}(x, n)^{\frac{1}{2}-\frac{i \lambda}{Q}}.$$ It is known  (\cite{RS, ACB97}) that for each fixed $x \in S,$ $\mathcal{P}_\lambda(x, \cdot) \in L^p(N)$ for $1 \leq p \leq \infty$ if $\lambda = i \gamma_p \rho,$ where $\gamma_p= \frac{2}{p}-1.$
	A very special feature of $\mathcal{P}_\lambda(x,n)$ is that it is constant on the hypersurfaces $H_{n, a_t}=\{n \sigma(a_t n'): \, n' \in N\}.$ Here $\sigma$ is the geodesic inversion on $S,$ that is an involutive, measure-preserving, diffeomorphism which can be explicitly given by \cite{CDK91}:
	\begin{align*}
	    \sigma(X,Z, a_t) = \left( \left(e^t+\frac{|V|^2}{4} \right)^2+|Z|^2 \right)^{-1} \left( \left(- \left( e^t+\frac{|X|^2}{4} \right)+J_Z \right)X, -Z, a_t  \right).
	\end{align*}

	Let $\Delta_S$ be the Laplace-Beltrami operator on $S.$ Then for every fixed $n \in N,$ $\mathcal{P}_\lambda(x, n)$ is an eigenfunction of $\Delta_S$ with eigenvalue $-(\lambda^2+\frac{Q^2}{4})$ (see \cite{ACB97}).
	For a measurable function $f$ on $S,$ the (Helgason) Fourier transform is defined as
	$$\widetilde{f}(\lambda, n)= \int_S f(x)\, \mathcal{P}_\lambda(x, n) dx$$ whenever the integral converges. For $f \in C_c^\infty(S),$ the following inversion formula holds (\cite[Theorem 4.4]{ACB97}):
	$$f(x)= C \int_{\mathbb{R}} \int_N \widetilde{f}(\lambda, n)\,\mathcal{P}_{-\lambda}(x, n) |c(\lambda)|^{-2} \, d\lambda dn.$$ The authors also proved that the (Helgason) Fourier transform extends to an isometry from $L^2(S)$ onto the space $L^2(\mathbb{R}_+ \times N, |c(\lambda)|^{-2}d\lambda dn).$ In fact they have the precise value of constants, we refer the reader to \cite{ACB97}. The following estimates for the function $|c(\lambda)|$ hold:
	$$|c(\lambda)|^{-2}= \begin{cases} |s|^2\quad & |s| \leq 1 \\ |s|^{d-1} \quad& |s| >1\end{cases}$$ for all $\lambda \in \mathbb{R}$ (e. g. see \cite[Theorem 1.14]{MV}).
	In \cite[Theorem 4.6]{RS}, the authors proved the following version of the Hausdorff-Young inequality:
	For $1 \leq p \leq 2$ we have 
	\begin{align}
	    \left(\int_{\mathbb{R}} \int_{N} |\widetilde{f}(\lambda+i \gamma_{p'} \rho, n)|^{p'} dn\, |c(\lambda)|^{-2} d\lambda \right)^{\frac{1}{p'}} \leq 
	C_p \|f\|_p.
	\end{align}
	Let $\tilde{\mu}$ be the metric induced by the canonical  left invariant Riemannian structure on $S$ and let $e$ be the identity element of $S.$  A function $f$ on $S$ is called {\it radial} if for all $x, y \in S ,$ $f(x)=f(y)$ if $\tilde{\mu}(x,e)=\tilde{\mu}(y,e).$ Note that radial functions on $S$ can be identified with the functions $f=f(r)$ of the geodesic distance  $r=\tilde{\mu}(x, e) \in [0, \infty)$ to the identity.  It is clear that $\tilde{\mu}(a_t, e)=|t|$ for $t \in \mathbb{R}.$ At times, for any radial function $f$ we use the notation $f(a_t)=f(t).$ For any function space $\mathcal{F}(S)$ on $S$, the subspace of radial functions will be denoted by $\mathcal{F}(S)^\#.$ 
	The elementary spherical function $\phi_\lambda(x)$ is defined by 
	 $$\phi_\lambda(x) :=\int_N \mathcal{P}_\lambda(x, n) \mathcal{P}_{-\lambda}(x, n)\, dn.  $$
	 It follows (\cite{Anker96, ACB97}) that $\phi_\lambda$ is a radial eigenfunction of the Laplace-Beltrami operator $\Delta_S$ of $S$ with eigenvalue $-(\lambda^2+\frac{Q^2}{4})$ such that $\phi_\lambda(x)=\phi_{-\lambda}(x),\,\, \phi_\lambda(x)=\phi_\lambda(x^{-1})$ and $\phi_\lambda(e)=1.$ 
	 In \cite{Anker96}, the authors showed that the radial part (in geodesic polar coordinates) of the Laplace-Beltrami operator $\Delta_S$ given by 
	 $$\textnormal{rad}\, \Delta_S= \frac{\partial^2}{\partial t}+\{ \frac{m+l}{2} \coth{ \frac{t}{2}}+\frac{k}{2} \tanh{\frac{t}{2}} \} \frac{\partial}{\partial t},$$
	 is (by subtituting $r=\frac{t}{2}$) equal to  $\frac{1}{4} \mathcal{L}_{\alpha, \beta}$  with indices $\alpha=\frac{m+l+1}{2}$ and $\beta=\frac{l-1}{2},$ where $\mathcal{L}_{\alpha, \beta}$ is the Jacobi operator studied by Koornwinder \cite{Koorn} in detail. It is worth noting that we are in the ideal situation of the Jacobi analysis with $\alpha>\beta>\frac{-1}{2}.$ In fact, the Jacobi functions $\phi_\lambda^{\alpha, \beta}$ and elementary spherical functions $\phi_\lambda$ are related as  (\cite{Anker96}): $\phi_\lambda(t)=\phi_{2 \lambda}^{\alpha, \beta}(\frac{t}{2}).$ 
	As consequence of this relation, the following estimates for the elementary spherical functions hold true (see \cite{Platonov}).
	\begin{lem} \label{estijac}  The following inequalities are valid for spherical functions $\phi_\lambda(t)\,\,(t, \lambda \in \mathbb{R}_+):$
	\begin{itemize}
	    \item $|\phi_\lambda(t)| \leq 1.$
	    \item $|1-\phi_\lambda(t)| \leq \frac{t^2}{2} (4 \lambda^2+\frac{Q^2}{4}).$
	    \item There exists a constant $c>0,$ depending only on $\lambda,$ such that $|1-\phi_\lambda(t)| \geq c$ for  $\lambda t \geq 1.$ 
	\end{itemize}
	\end{lem} 
	Now, we define the spherical (or Harish Chandra) transform of an integrable radial function $f$ on $S$ by 
	$$\mathcal{H}(f)(\lambda):= \int_S f(x) \phi_\lambda(x)\, dx.$$
	For $f \in C_c^\infty(S)^\#,$ the following inversion formula holds:
	$$\mathcal{H}^{-1}(f)(x):= c_S \int_{0}^\infty \mathcal{H}f(\lambda) \phi_\lambda(x)\,|c(\lambda)|^{-2}\,d\lambda,$$ where $c_S$ depends only on $m$ and $l,$ and $c(\lambda)$ is the Harish-Chandra function. 
	Moreover, the following Plancherel formula holds:
	   $$\int_S |f(x)|^2 dx = c_S\int_{\mathbb{R}_+} |\mathcal{H}f(\lambda)|^2\, |c(\lambda)|^{-2} d\lambda.$$
	   The spherical Fourier transform extends to an isometry from $L^2(S)^\#$ to $L^2(\mathbb{R}_+, c_S |c(\lambda)|^{-2}d\lambda).$ For details on the spherical analysis on harmonic $NA$ group we refer to \cite{Anker96}. 
	   
	   Let $\sigma_t$ be the normalized surface measure of the geodesic sphere of radius $t$. Then $\sigma_t$ is a nonnegative radial measure. The spherical mean operator $M_t$ on a suitable function space on $S$ is defined by $M_tf: =f*\sigma_t.$  It can be noted that $M_tf(x)=\mathcal{R}(f^x)(t)$, where $f^x$ denotes  the right translation of function $f$ by $x$ and $\mathcal{R}$ is the radialization operator defined, for suitable function $f,$ by
	   $$\mathcal{R}f(x)=\int_{S_\nu} f(y)\,d\sigma_{\nu}(y),$$ where $\nu=r(x)= \mu(C(x), 0),$ here $C$ is the Cayley transform, and $d\sigma_\nu$ is the normalized surface measure induced by the left invariant Riemannian metric on the geodesic sphere $S_\nu=\{y \in S: \mu(y, e)=\nu \}.$ It is easy to see that  $\mathcal{R}f$ is a radial function and for any radial function $f,$ $\mathcal{R}f=f.$ Consequently, for a radial function $f,$\, $M_tf$ is the usual translation of $f$ by $t.$  In \cite{KRS}, the authors proved that $M_t$ is a bounded linear operator on $L^p(S)$ and for a suitable function $f$ on $S,$ $\widetilde{M_t f}(\lambda, n)= \widetilde{f}(\lambda, n) \phi_\lambda(t)$ whenever both make sense. Also, $M_tf$ converges to $f$ as $t \rightarrow 0.$ The following result was proved by Kumar et al. in \cite[Theorem 4.7 (a)]{KRS} which was initially conjectured by Bray and Pinsky \cite[Conjecture 16]{BrayPinsky}. 
	\begin{thm} \label{KRSconje} Let $1 <p \leq 2$ and $p \leq q \leq p'.$ Then for $f \in L^p(S)$
	$$\left( \int_{\mathbb{R}} \min \{1, (\lambda t)^{2p'}\} \left( \int_N |\widetilde{f}(\lambda+i \gamma_q \rho, n)|^q \,dn \right)^{p'/q} |c(\lambda)|^{-2} d\lambda \right)^{\frac{1}{p'}} \leq C_{p,q} \|M_tf-f\|_{p}.$$
 	\end{thm}

  For small $|\lambda|,$ with $q=p',$ we have 

\begin{equation} \label{Inte1}
     \int_{|\lambda|<\frac{1}{t}} |\lambda|^{2p'}  \int_N |\widetilde{f}(\lambda+i \gamma_{p'} \rho, n)|^{p'} \,dn |c(\lambda)|^{-2} d\lambda \leq C_{p}^{p'} \left( \frac{\|M_tf-f\|_{p}}{t^2} \right)^{p'}.
\end{equation}

		For notational convenience, we take integral over an empty set to be
zero.
	\section{Titchmarsh theorems for (Helgason) Fourier transform of H$\ddot{\text{o}}$lder-Lipschitz functions on harmonic $NA$ groups }
	This section is devoted to the investigation of Titchmarsh theorems on H$\ddot{\text{o}}$lder-Lipschitz functions on harmonic $NA$ groups. In this plan of action we first study the growth property of Fourier transform of $L^2$-Lipschitz function. Then we present some applications of it in the context of multiplier theorems and embedding theorems of Lipschitz-Sobolev spaces. After that, we prove an analogue of the second Titchmarsh theorem improving the range of Hausdorff-Young inequality for H$\ddot{\text{o}}$lder-Lipschitz functions on harmonic $NA$ groups.

	\subsection{Growth properties of Fourier transform of H$\ddot{\text{o}}$lder-Lipschitz functions on harmonic $NA$ groups} \label{Subs3.1}
	The following theorem generalizes one result found in \cite{Platonov} for rank one noncompact symmetric spaces. Recently, this result was extended by Bray \cite{Bray} in higher dimensional Euclidean spaces, which dates back to Titchmarsh \cite[Theorem 17]{Titch}  in the one dimensional setting (see Theorem B).

	\begin{thm} \label{Titch1} Let $0<\alpha \leq 1$ and $f \in L^2(S).$ Then $f \in \text{Lip}_S(\alpha; 2)$ if and only if 
	$$ \int_{\lambda >\frac{1}{t}} \int_N |\widetilde{f}(\lambda, n)|^2 dn |c(\lambda)|^{-2} \, d \lambda = O(t^{2 \alpha})\,\,\,\,\text{as}\,\, t \rightarrow 0.$$	\end{thm}
	\begin{proof}
	
	Let us assume that $f \in \text{Lip}_S(\alpha; 2)$, i.e., 
	\begin{equation} \label{lic}
	    \|M_tf-f\|_2= O(t^\alpha)\,\,\,\text{as}\,\, t \rightarrow 0.
	\end{equation}
	By the Plancherel theorem and the fact that $\widetilde{M_t f}(\lambda, n)= \phi_\lambda(a_t) \widetilde{f}(\lambda, n)$ we conclude that
	\begin{align*}
	    \|M_tf-f\|_2^2 &= \|\widetilde{M_t f-f}\|^2_{L^2(\mathbb{R}_+ \times N, |c(\lambda)|^{-2} d\lambda \, dn)} \\&=  \int_{0}^\infty \int_N |1- \phi_\lambda(a_t) |^2 |\widetilde{f}(\lambda, n)|^2 |c(\lambda)|^{-2} d\lambda\, dn \\&=  \int_{0}^\infty  |1- \phi_\lambda(a_t) |^2 G(\lambda) |c(\lambda)|^{-2} d\lambda,
	\end{align*}
where $ G(\lambda)=	\int_N |\widetilde{f}(\lambda, n)|^2 dn.$
Now, by \eqref{lic} it follows that 
$$  \int_{0}^\infty  |1- \phi_\lambda(a_t) |^2 G(\lambda) |c(\lambda)|^{-2} d\lambda = O\left( t^{2\alpha}\right)\,\,\text{as}\, t \rightarrow 0.$$

Now, if $\lambda \in [\frac{1}{t}, \frac{2}{t}]$ then $ \lambda t \geq 1$ and by Lemma \ref{estijac}(iii) it follows that there exists a constant $c>0$ such that 
$ \frac{|1-\phi_\lambda(a_t)|^2}{c^2} \geq 1.$
Using this, we have 
\begin{align*}
    \int_{\frac{1}{t}}^{\frac{2}{t}} G(\lambda)\, |c(\lambda)|^{-2}\, d\lambda &\leq \frac{1}{c^2}  \int_{\frac{1}{t}}^{\frac{2}{t}} |1-\phi_\lambda(a_t)|^2 G(\lambda)\, |c(\lambda)|^{-2}\, d\lambda \\ &\leq \frac{1}{c^2}  \int_0^\infty |1-\phi_\lambda(a_t)|^2 G(\lambda)\, |c(\lambda)|^{-2}\, d\lambda = O\left( t^{2\alpha}\right)\,\,\text{as}\, t \rightarrow 0.
\end{align*}
By setting $r=\frac{1}{t},$ we may rewrite the above inequality as 
$$\int_{r}^{2r} G(\lambda)\, |c(\lambda)|^{-2}\, d\lambda \leq C' \left( r^{-2 \alpha} \right).$$ 
Consequently, we have 
\begin{align*}
    \int_{r}^{\infty} G(\lambda)\, |c(\lambda)|^{-2}\, d\lambda & = \left[ \int_{r}^{2r}+\int_{2r}^{4r}+\int_{4r}^{8r}+\cdots \right]G(\lambda)\, |c(\lambda)|^{-2}\, d\lambda  \\& \leq C' \left(r^{-2 \alpha}+(2r)^{-2 \alpha} +(4r)^{-2 \alpha}+\cdots \right) \\&\leq C' r^{-2 \alpha} \left[1+2^{-2\alpha}+ (2^{-2\alpha})^2+(2^{-2\alpha})^3+\cdots \right] \\&= C' (1-2^{-2\alpha})^{-1} r^{-2 \alpha}.
\end{align*}
Therefore, we have \begin{equation*}
        \int_{r}^\infty \int_N |\widetilde{f}(\lambda, n)|^2 dn\,|c(\lambda)|^{-2} \, d\lambda = O\left( r^{-2\alpha}\right)\,\,\,\text{as}\,\, r \rightarrow \infty.
    \end{equation*}
    Now we will prove the converse implication. Assume that 
    \begin{equation}
        \int_{r}^\infty \int_N |\widetilde{f}(\lambda, n)|^2 dn \, d\lambda = O\left(r^{-2\alpha}\right)\,\,\,\text{as}\,\, r \rightarrow \infty.
    \end{equation}
As earlier we set $G(\lambda)= \int_N |\widetilde{f}(\lambda, n)|^2\, dn.$ Then  we get
\begin{align*}
    \int_{r}^{2r} G(\lambda)\, |c(\lambda)|^{-2}d\lambda 
     \leq C r^{-2\alpha}.
\end{align*}
As a consequence, it follows that 
\begin{align*}
    \int_r^\infty G(\lambda) |c(\lambda)|^{-2}\, d\lambda &= \sum_{j=0}^\infty \int_{2^j r}^{2^{j+1}r} G(\lambda)\, |c(\lambda)|^{-2}\, d\lambda  \leq C \sum_{j=0}^{\infty}(2^{-2\alpha})^k r^{-2 \alpha} \leq C r^{-2 \alpha},
\end{align*} where $C$ is a positive constant.
Further, we have 
\begin{align*}
	    \|M_tf-f\|_{L^2(S)}^2&= \|\widetilde{M_t f-f}\|^2_{L^2(\mathbb{R}_+ \times N, |c(\lambda)|^{-2} d\lambda \, dn)} \\&=  \int_{0}^\infty \int_N |1- \phi_\lambda(a_t) |^2 |\widetilde{f}(\lambda, n)|^2 |c(\lambda)|^{-2} d\lambda\, dn \\&=  \int_{0}^\infty  |1- \phi_\lambda(a_t) |^2 G(\lambda) |c(\lambda)|^{-2} d\lambda \\&= \int_{0}^{\frac{1}{t}} |1- \phi_\lambda(a_t) |^2 G(\lambda) |c(\lambda)|^{-2} d\lambda+ \int_{\frac{1}{t}}^{\infty} |1- \phi_\lambda(a_t) |^2 G(\lambda) |c(\lambda)|^{-2} d\lambda \\&= 
	    I_1+I_2.
	\end{align*}
	For second integral $I_2,$ by using $|\phi_\lambda(a_t)| \leq 1$ for $\lambda \in \mathbb{R}_+$ we have, with $r=\frac{1}{t},$ 
	\begin{align*}
	    I_2 &=\int_{\frac{1}{t}}^\infty |1- \phi_\lambda(a_t) |^2 G(\lambda) |c(\lambda)|^{-2} d\lambda \leq  4 \int_{\frac{1}{t}}^\infty G(\lambda)\, |c(\lambda)|^{-2} d\lambda= 4 C r^{-2\alpha}
	\end{align*}
	and so 
	\begin{equation}\label{I_2vi}
	    I_2= O \left(r^{-2\alpha}\right).
	\end{equation}
	Next, to estimate $I_1,$ we will use the facts that $|\phi_\lambda(a_t)| \leq 1$ and $|1-\phi_\lambda(a_t)| \leq t^2 (\lambda^2+\frac{Q^2}{4})$ for $\lambda,\, t \in \mathbb{R}_+$ (see Lemma \ref{estijac}), so that 
	\begin{align*}
	    I_1 &= \int_{0}^{\frac{1}{t}} |1- \phi_\lambda(a_t) |^2 G(\lambda) |c(\lambda)|^{-2} d\lambda = \int_{0}^{\frac{1}{t}} |1- \phi_\lambda(a_t) |\,|1- \phi_\lambda(a_t) | G(\lambda) |c(\lambda)|^{-2} d\lambda \\ & \leq \int_{0}^{\frac{1}{t}} (1+| \phi_\lambda(a_t)| )\,|1- \phi_\lambda(a_t) | G(\lambda) |c(\lambda)|^{-2} d\lambda \\& \leq 2 \int_{0}^{\frac{1}{t}} |1- \phi_\lambda(a_t) | G(\lambda) |c(\lambda)|^{-2} d\lambda \leq  2 t^2 \int_{0}^{\frac{1}{t}} \left(\lambda^2+\frac{Q^2}{4}\right)\, G(\lambda)\, |c(\lambda)|^{-2}\, d\lambda.
	\end{align*} For a while, we put 
	$F(r)=\int_r^\infty G(\lambda) d\mu(\lambda),$ where $d\mu(\lambda)= |c(\lambda)|^{-2} d\lambda.$ Then we have, as $0 \leq \lambda \leq \frac{1}{t}=r,$ that
    \begin{align}
        I_1 & \leq  2 t^2 \int_{0}^{\frac{1}{t}} -\left(r^{2}+\frac{Q^2}{4}\right) F'(r)\, dr \leq 2 t^2 \int_{0}^{\frac{1}{t}} -r^{2} F'(r)\, dr  \\&= 2t^2 \left( \frac{-1}{t^2} F\left(\frac{1}{t}\right)+2 \int_0^{\frac{1}{t}} r F(r)\, dr \right)= -2F\left(\frac{1}{t}\right)+4t^2 \int_0^{\frac{1}{t}} r F(r)\, dr \\&\leq 4t^2 \int_0^{\frac{1}{t}} r F(r)\, dr. \end{align}
    Since $F(r)= \int_r^\infty G(\lambda) d\mu(\lambda)= O\left( r^{-2 \alpha}\right)$ we get 
    $$I_1 \leq C t^2 \int_0^{\frac{1}{t}} r\left(r^{-2 \alpha}\right)\, dr \leq  C t^{2\alpha}, $$ where $C$ is a positive constant. Hence, by combining $I_1$ and $I_2$ we get 
    $$\|M_tf-f\|_{L^2(S)}^2= \left( t^{2\alpha}\right)\,\,\,\text{as}\, t \rightarrow 0.$$
    Hence, $f \in \text{Lip}_S(\alpha,2).$
	\end{proof}

\subsection{Applications}
This subsection is devoted to the applications of the result presented in Subsection \ref{Subs3.1}. Specifically, we formulate the following corollary for the regularity of Fourier multipliers of H\"older-Lipschitz  spaces on harmonic $NA$ groups. This corollary complements the other results available in the literature (e.g. see \cite{Anker96}) for boundedness of Fourier multipliers on harmonic $NA$ groups. 

	\begin{cor} \label{coro3.6}
	 Let $0 \leq \gamma <1$ and let $h$ be an even measurable function on $\mathbb{R}$ such that 
	 $$|h(\lambda)|\leq C \langle \lambda \rangle^{-\gamma},$$ where $\langle \lambda \rangle = (\lambda^2+\frac{Q^2}{2})^{\frac{1}{2}}.$
	 Let $A$ be the Fourier multiplier with symbol $h$, i.e., given by $\widetilde{Af}(\lambda, n)= h(\lambda) \widetilde{f}(\lambda, n)$ for all $\lambda \in \mathbb{R}_+$ and $n \in N.$ Then 
	 $$A:\text{Lip}_S(\alpha; 2) \rightarrow \text{Lip}_S(\alpha+\gamma;2)$$ is bounded for all $\alpha$ such that $0<\alpha<1-\gamma.$
	\end{cor}
	\begin{proof}
	Let $f \in \text{Lip}_S(\alpha; 2).$ Then by Theorem \ref{Titch1} we have 
	\begin{align*}
	    \int_{\lambda >\frac{1}{t}} \int_N |\widetilde{Af}(\lambda, n)|^2 dn |c(\lambda)|^{-2} \, d \lambda &= \int_{\lambda >\frac{1}{t}} \int_N |h(\lambda)\widetilde{f}(\lambda, n)|^2 dn |c(\lambda)|^{-2} \, d \lambda \\& \leq C \int_{\lambda >\frac{1}{t}} \int_N \langle \lambda \rangle^{-2\gamma}|\widetilde{f}(\lambda, n)|^2 dn |c(\lambda)|^{-2} \, d \lambda \\& \leq  C t^{2\gamma} \int_{\lambda >\frac{1}{t}} \int_N |\widetilde{f}(\lambda, n)|^2 dn |c(\lambda)|^{-2} \, d \lambda \\& = O(t^{2(\alpha+\gamma)})\,\,\,\text{as}\,\,t \rightarrow 0.
	\end{align*}
	
	Again by Theorem \ref{Titch1} this implies that $A:\text{Lip}_S(\alpha; 2) \rightarrow \text{Lip}_S(\alpha+\gamma;2)$ is bounded for all $\alpha>0$ such that $\alpha+\gamma<1.$
	\end{proof}
	
	As an example we consider the Lipschitz regularity for the Laplace-Beltrami operator $(-\Delta_S)^{\frac{-\gamma}{2}}.$ First we observe that if $A:=(-\Delta_S)^{\frac{-\gamma}{2}}$ with $0 \leq \gamma <1,$ by Corollary \ref{coro3.6} we get 
	$$\|(-\Delta_S)^{\frac{-\gamma}{2}} f\|_{\textnormal{Lip}_S(\alpha+\gamma; 2)} \leq C \|f\|_{\textnormal{Lip}_S(\alpha; 2)}$$ for all $\alpha$ such that $0<\alpha<1-\gamma.$
	Hence, \begin{equation} \label{sobolev}
	    \| f\|_{\textnormal{Lip}_S(\alpha+\gamma; 2)} \leq C \|(-\Delta_S)^{\frac{\gamma}{2}}f\|_{\textnormal{Lip}_S(\alpha; 2)}.
	\end{equation}
	Now, we can introduce the Sobolev-Lipschitz space $H^\gamma \textnormal{Lip}_S(\alpha; 2)$ for every $0 \leq  \gamma <1$ and $0< \alpha<1-\gamma$ by 
	$$H^\gamma \textnormal{Lip}_S(\alpha; 2):= \left\{ f\in \mathcal{D}'(S):  (-\Delta_S)^{\frac{\gamma}{2}} f \in \textnormal{Lip}_S(\alpha; 2) \right\}.$$
	From \eqref{sobolev} we have the following corollary:
	\begin{cor}
	 For every $0 \leq \gamma <1$ and $0 <\alpha<1-\gamma$ we have the continuous embedding 
	 $$H^\gamma \textnormal{Lip}_S(\alpha; 2) \hookrightarrow \textnormal{Lip}_S(\alpha+\gamma; 2).$$
	\end{cor}
	
	\subsection{An integrability result for  Fourier transform of H\"older-Lipschitz functions on harmonic $NA$ groups} The following result is  a general integrabilty theorem for the function $F$ on $[0, \infty)$ defined by $$F(\lambda):=   \left(\int_N |\widetilde{f}(\lambda+i \gamma_{p'} \rho, n)|^{p'} \,dn \right)^{\frac{1}{p'}},\,\, 1<p \leq 2,
$$ where $\gamma_{p'}=\frac{2}{p'}-1=-\gamma_p.$ This theorem is an extension of a result of Bray (\cite{Bray}) on $\mathbb{R}^n$ to harmonic $NA$ groups, in particular to rank one noncompact symmetric spaces.  We prove this result by applying the similar technique as in \cite{Bray}. In our case, an inequality \eqref{Inte1} obtained as consequence of a result (Theorem \ref{KRSconje}) proved by Kumar et al. \cite{KRS} plays a pivotal role.  
	
\begin{thm} \label{Titch2} Let $S$ be a harmonic NA group of dimension $d$. Let $0 <\alpha \leq 1,\, 1<p \leq 2,$ and let $p'$ be  such that $\frac{1}{p}+\frac{1}{p'}=1.$
    Let $f \in \text{Lip}_S(\alpha; p).$ Then
    the function $F$ defined by 
$$F(\lambda):=   \left(\int_N |\widetilde{f}(\lambda+i \gamma_{p'} \rho, n)|^{p'} \,dn \right)^{\frac{1}{p'}},\,\, 1<p \leq 2,
$$
    belongs to $ L^\beta((0, \infty), |c(\lambda)|^{-2}\, d \lambda)$ provided that 
    \begin{align} \label{beta}
        \frac{dp}{dp+\alpha p-d} < \beta \leq p'.
    \end{align}
\end{thm}

\begin{proof} For $f \in \text{Lip}_S(\alpha; p)$ we have that $$\|M_tf-f\|_{p} = O(t^\alpha)$$ as $t \rightarrow 0.$ Therefore, by setting $\Lambda= \frac{1}{t}$ from inequality \eqref{Inte1} we get that
$$\int_{1}^{\Lambda} \lambda^{2p'} F(\lambda)^{p'} |c(\lambda)|^{-2}d\lambda  \leq C_{p}^{p'} \left( \frac{\|M_tf-f\|_{p}}{t^2} \right)^{p'} \leq C \Lambda^{(2-\alpha)p'},$$
since $\|M_tf-f\|_{p} = O(t^\alpha)$ we get, by putting $t=\frac{1}{\Lambda},$ that  $$\left( \frac{\|M_tf-f\|_{p}}{t^2} \right)^{p'} = \left( O(\Lambda^{-\alpha}) O(\Lambda^2) \right)^{p'}=O(\Lambda^{(2-\alpha)p'}).$$

Take $\beta < p'$ and let 
$$G(\Lambda)= \int_1^\Lambda \lambda^{2 \beta} F(\lambda)^\beta |c(\lambda)|^{-2} d\lambda.$$

By H\"older's inequality we conclude that by using the estimate $|c(\lambda)|^{-2} \lesssim |\lambda|^{d-1}$ as $|\lambda|>1,$ 
\begin{align} \label{t9}
    G (\Lambda) & = \int_{1}^{\Lambda} \left[ \lambda^2\, F(\lambda) \right]^\beta |c(\lambda)|^{-2(\frac{\beta}{p'})} |c(\lambda)|^{-2(1-\frac{\beta}{p'})} d\lambda \nonumber \\ & \leq 
    C \left( \int_{1}^\Lambda \lambda^{2p'} F(\lambda)^{p'} |c(\lambda)|^{-2} d\lambda \right)^{\frac{\beta}{p'}} \left( \int_{1}^\Lambda |c(\lambda)|^{-2}\, d\lambda \right)^{1-\frac{\beta}{p'}} \nonumber 
    \\& \leq C \left( \int_{1}^\Lambda \lambda^{2p'} F(\lambda)^{p'} |c(\lambda)|^{-2} d\lambda \right)^{\frac{\beta}{p'}} \left( \int_{1}^\Lambda |\lambda|^{d-1}\, d\lambda \right)^{1-\frac{\beta}{p'}} \nonumber \\ & \leq 
    C \Lambda^{(2-\alpha)\beta} \Lambda^{d(1-\frac{\beta}{p'})} = O(\Lambda^{2\beta+d-\alpha \beta-\frac{d\beta}{p'}}),
    \end{align}
where $C$ is a generic constant which can be different at each step. 

Now, integration by parts yield the following, 
\begin{align*}
    \int_{1}^\Lambda F(\lambda)^{\beta} |c(\lambda)|^{-2} d\lambda &= \int_{1}^{\Lambda} \lambda^{-2\beta}\, G'(\lambda)\, d\lambda \\&= \Lambda^{-2\beta} G(\Lambda)+ 2 \beta \int_1^{\Lambda} \lambda^{-2\beta-1} G(\lambda)\, d\lambda =I_1+I_2.
    \end{align*}
 The first term $I_1$ is bounded under the condition on $\beta$ by using the above estimate of $G.$ Indeed 
 $$I_1= \Lambda^{-2 \beta} G(\Lambda) \leq \Lambda^{-2 \beta} C \Lambda^{(2-\alpha)\beta} \Lambda^{d(1-\frac{\beta}{p'})} =C \Lambda^{-\alpha \beta+d-\frac{d\beta}{p'}} = O(1)$$ since $-\alpha \beta+d-\frac{d\beta}{p'}= \alpha \beta+d-d \beta (1-\frac{1}{p})= \frac{dp-\beta(dp+\alpha p-d)}{p}$ which is negative by using the condition of $\beta,$ i.e., $\beta> \frac{dp}{dp+\alpha p-d}.$
 
 Similarly for $I_2,$ we get, using the estimate of $G(\lambda),$ 
 \begin{align}
     I_2= \int_1^{\Lambda} \lambda^{-2\beta-1} G(\lambda)\, d\lambda \leq C \int_1^{\Lambda} \lambda^{-\alpha \beta +d-\frac{d\beta}{p'}-1} d\lambda = O(1)
 \end{align}
 because $-\alpha\beta+d-\frac{d \beta}{p'}-1= \frac{dp-\beta(d p+\alpha p-d)-p}{p} <0$ by using the fact that $\beta> \frac{dp}{dp+\alpha p-d},$ which is equivalent to $dp-\beta (dp+\alpha p-d)<0.$
 Therefore, $$\int_1^{\Lambda} F(\lambda)^\beta |c(\lambda)|^{-2} d\lambda =O(1)$$ provided that $d-\alpha \beta -d \beta +\frac{d \beta}{p} <0$ and hence $F \in L^\beta((0, \infty), |c(\lambda)|^{-2} d\lambda),$ since the above $O(1)$ are independent of $\Lambda.$ The proof is complete as conditions on $\beta$ are equivalent to \eqref{beta}.\end{proof}
	\section{Helgason transform of Dini-Lipschitz functions on harmonic NA groups}
	
	In this section, we discuss the result related with the  growth and integrability of Fourier trasform of  Dini-Lipschitz functions on harmonic $NA$ groups. This kind of results were first studied by Younis \cite{Younis} in the context of one dimensional Euclidean space. These results were extended by many researchers in different settings \cite{Younis1, Younis2, FBK, DDR}. In particular, for compact homogeneous manifolds it was shown in \cite{DDR}. In this direction, we state our first result below. 
	
\begin{thm} \label{Them4.1}
    Let $S$ be a harmonic $NA$ group of dimension $d$ and let $\alpha, \beta >0.$ Then the conditions
    \begin{equation} \label{Dinilip}
        \|M_t f-f\|_{L^2(S)}= O \left( \frac{t^\alpha}{(\log \frac{1}{t})^\beta}\right)\,\,\text{as}\quad t \rightarrow 0 
    \end{equation}
    and 
    \begin{equation} \label{dini5}
        \int_{r}^\infty \int_N |\widetilde{f}(\lambda, n)|^2 dn \, d\lambda = O\left( \frac{r^{-2\alpha-d+1}}{(\log r)^{2\beta}}\right)\,\,\,\text{as}\,\, r \rightarrow \infty
    \end{equation} are equivalent. 
\end{thm}	
	\begin{proof}
	Let us assume that \eqref{Dinilip} holds.  By the Plancherel theorem and the fact that $$\widetilde{M_t f}(\lambda, n)= \phi_\lambda(a_t) \widetilde{f}(\lambda, n)$$ we get 
	\begin{align*}
	    \|M_tf-f\|_2^2&= \|\widetilde{M_t f-f}\|^2_{L^2(\mathbb{R}_+ \times N, |c(\lambda)|^{-2} d\lambda \, dn)} \\&=  \int_{0}^\infty \int_N |1- \phi_\lambda(a_t) |^2 |\widetilde{f}(\lambda, n)|^2 |c(\lambda)|^{-2} d\lambda\, dn \\&=  \int_{0}^\infty  |1- \phi_\lambda(a_t) |^2 G(\lambda) |c(\lambda)|^{-2} d\lambda,
	\end{align*}
where $ G(\lambda)=	\int_N |\widetilde{f}(\lambda, n)|^2 dn.$
Now, by \eqref{Dinilip} it follows that 
$$  \int_{0}^\infty  |1- \phi_\lambda(a_t) |^2 G(\lambda) |c(\lambda)|^{-2} d\lambda = O\left( \frac{t^{2\alpha}}{(\log \frac{1}{t})^{2\beta}}\right)\,\,\text{as}\, t \rightarrow 0.$$

Now, if $\lambda \in [\frac{1}{t}, \frac{2}{t}]$ then $|\lambda t| \geq 1$ and by Lemma \ref{estijac}(iii) it follows that there exists a constant $c>0$ such that 
$$ \frac{|1-\phi_\lambda(a_t)|^2}{c^2} \geq 1.$$
Using this, we have 
\begin{align*}
    \int_{\frac{1}{t}}^{\frac{2}{t}} G(\lambda)\, |c(\lambda)|^{-2}\, d\lambda &\leq \frac{1}{c^2}  \int_{\frac{1}{t}}^{\frac{2}{t}} |1-\phi_\lambda(a_t)|^2 G(\lambda)\, |c(\lambda)|^{-2}\, d\lambda \\ &\leq \frac{1}{c^2}  \int_0^\infty |1-\phi_\lambda(a_t)|^2 G(\lambda)\, |c(\lambda)|^{-2}\, d\lambda\\&= O\left( \frac{t^{2\alpha}}{(\log \frac{1}{t})^{2\beta}}\right)\,\,\text{as}\, t \rightarrow 0.
\end{align*}
By setting $r=\frac{1}{t},$ we may write the above inequality as 
$$\int_{r}^{2r} G(\lambda)\, |c(\lambda)|^{-2}\, d\lambda \leq C' \left( \frac{r^{-2 \alpha}}{ (\log r)^{2 \beta}} \right),$$ where $C'>0$ is a generic constant which can be different for different inequalites below. 
Consequently, we have 
\begin{align*}
    \int_{r}^{\infty} G(\lambda)\, |c(\lambda)|^{-2}\, d\lambda & = \left[ \int_{r}^{2r}+\int_{2r}^{4r}+\int_{4r}^{8r}+\cdots \right]G(\lambda)\, |c(\lambda)|^{-2}\, d\lambda  \\& \leq C' \left(\frac{r^{-2 \alpha}}{ (\log r)^{2 \beta}}+\frac{(2r)^{-2 \alpha}}{ (\log 2r)^{2 \beta}}+\frac{(4r)^{-2 \alpha}}{ (\log 4r)^{2 \beta}}+ \cdots \right) \\&\leq C' \frac{r^{-2 \alpha}}{ (\log r)^{2 \beta}} \left[1+2^{-2\alpha}+ (2^{-2\alpha})^2+(2^{-2\alpha})^3+\cdots \right] \\&= C' (1-2^{-2\alpha})^{-1} \frac{r^{-2 \alpha}}{ (\log r)^{2 \beta}}.
\end{align*}
By using the estimate $|c(\lambda)|^{-2} \geq c_1 \lambda^{d-1}$ for $\lambda \rightarrow \infty$ we get 
$$\int_{r}^{\infty} G(\lambda)\, d\lambda= c_1 r^{-d+1}\int_{r}^{\infty} G(\lambda)\, |c(\lambda)|^{-2}\, d\lambda= c_1 C'(1-2^{-2\alpha})^{-1} r^{-d+1}  \frac{r^{-2 \alpha}}{ (\log r)^{2 \beta}} $$ and therefore, we have 

\begin{equation}
        \int_{r}^\infty \int_N |\widetilde{f}(\lambda, n)|^2 dn \, d\lambda = O\left( \frac{r^{-2\alpha-d+1}}{(\log r)^{2\beta}}\right)\,\,\,\text{as}\,\, r \rightarrow \infty.
    \end{equation}
    
    Now we will prove the converse implication. Assume that \eqref{dini5} holds, i.e., 
    \begin{equation}
        \int_{r}^\infty \int_N |\widetilde{f}(\lambda, n)|^2 dn \, d\lambda = O\left( \frac{r^{-2\alpha-d+1}}{(\log r)^{2\beta}}\right)\,\,\,\text{as}\,\, r \rightarrow \infty.
    \end{equation}
As earlier we set $G(\lambda)= \int_N |\widetilde{f}(\lambda, n)|^2\, dn.$ Then by using the estimate $|c(\lambda)|^{-2} \leq c_2 |\lambda|^{d-1}$ we get
\begin{align*}
    \int_{r}^{2r} G(\lambda)\, |c(\lambda)|^{-2}d\lambda &\leq c_2 \int_{r}^{2r} G(\lambda)\, \lambda^{d-1} d\lambda  \leq c_2 (2r)^{d-1} \int_{r}^{2r} G(\lambda)\, d\lambda 
    \\&\leq C r^{d-1} \int_r^\infty G(\lambda)\, d\lambda
     \leq C \frac{r^{-2\alpha}}{(\log r)^{2 \beta}}.
\end{align*}
As a consequence, it follows that 
\begin{align*}
    \int_r^\infty G(\lambda) |c(\lambda)|^{-2}\, d\lambda &= \sum_{j=0}^\infty \int_{2^j r}^{2^{j+1}r} G(\lambda)\, |c(\lambda)|^{-2}\, d\lambda \\& \leq C \sum_{j=0}^{\infty}\frac{(2^{-2\alpha})^k r^{-2 \alpha}}{(\log r)^{2 \beta}} \leq C \frac{r^{-2 \alpha}}{(\log r)^{2\beta}},
\end{align*} where $C$ is a positive constant.
Further, we have 
\begin{align*}
	    \|M_tf-f\|_{L^2(S)}^2&= \|\widetilde{M_t f-f}\|^2_{L^2(\mathbb{R}_+ \times N, |c(\lambda)|^{-2} d\lambda \, dn)} \\&=  \int_{0}^\infty \int_N |1- \phi_\lambda(a_t) |^2 |\widetilde{f}(\lambda, n)|^2 |c(\lambda)|^{-2} d\lambda\, dn \\&=  \int_{0}^\infty  |1- \phi_\lambda(a_t) |^2 G(\lambda) |c(\lambda)|^{-2} d\lambda \\&= \int_{0}^{\frac{1}{t}} |1- \phi_\lambda(a_t) |^2 G(\lambda) |c(\lambda)|^{-2} d\lambda+ \int_{\frac{1}{t}}^{\infty} |1- \phi_\lambda(a_t) |^2 G(\lambda) |c(\lambda)|^{-2} d\lambda \\&= 
	    I_1+I_2.
	\end{align*}
	For the second integral $I_2,$ by using $|\phi_\lambda(a_t)| \leq 1$ for $\lambda \in \mathbb{R}_+,$ we have 
	\begin{align*}
	    I_2 &=\int_{\frac{1}{t}}^{\infty} |1- \phi_\lambda(a_t) |^2 G(\lambda) |c(\lambda)|^{-2} d\lambda \\&\leq  4 \int_{\frac{1}{t}}^\infty G(\lambda)\, |c(\lambda)|^{-2} d\lambda= 4 C \frac{r^{-2\alpha}}{(\log r)^{2 \beta}}
	\end{align*}
	and so 
	\begin{equation}\label{I_2vi}
	    I_2= O \left( \frac{r^{-2\alpha}}{(\log r)^{2 \beta}}\right).
	\end{equation}
	Next, to estimate $I_1,$ we will use the facts that $|1-\phi_\lambda(t)| \leq 1+|\phi_\lambda(t)| \leq 2$ and  $|1-\phi_\lambda(t)| \leq t^2 (\lambda^2+\frac{Q^2}{4})$ for $\lambda,\, t \in \mathbb{R}_+$ (see Lemma \ref{estijac}), to obtain 
	\begin{align*}
	    I_1 &= \int_{0}^{\frac{1}{t}} |1- \phi_\lambda(a_t) |^2 G(\lambda) |c(\lambda)|^{-2} d\lambda  = \int_{0}^{\frac{1}{t}} |1- \phi_\lambda(a_t) |\,|1- \phi_\lambda(a_t) | G(\lambda) |c(\lambda)|^{-2} d\lambda \\ & \leq \int_{0}^{\frac{1}{t}} (1+| \phi_\lambda(a_t)| )\,|1- \phi_\lambda(a_t) | G(\lambda) |c(\lambda)|^{-2} d\lambda \\&  \leq 2\int_{0}^{\frac{1}{t}} |1- \phi_\lambda(a_t) | G(\lambda) |c(\lambda)|^{-2} d\lambda  \leq  2 t^2 \int_{0}^{\frac{1}{t}} (\lambda^2+\frac{Q^2}{4})\, G(\lambda)\, |c(\lambda)|^{-2}\, d\lambda.
	\end{align*} For a while, we put 
	$F(r)=\int_r^\infty G(\lambda) d\mu(\lambda),$ where $d\mu(\lambda)= |c(\lambda)|^{-2} d\lambda.$ Then we have 
    \begin{align}
        I_1 & = 2 t^2 \int_{0}^{\frac{1}{t}} -(r^{2}+\frac{Q^2}{4}) F'(r)\, dr \leq 2 t^2 \int_{0}^{\frac{1}{t}} -r^{2} F'(r)\, dr  \\&= 2t^2 \left( \frac{-1}{t^2} F(\frac{1}{t})+2 \int_0^{\frac{1}{t}} r F(r)\, dr \right)= -F(\frac{1}{t})+4t^2 \int_0^{\frac{1}{t}} r F(r)\, dr \\&\leq 4t^2 \int_0^{\frac{1}{t}} r F(r)\, dr. 
    \end{align}
    Since $F(r)= \int_r^\infty G(\lambda) d\mu(\lambda)= O\left( \frac{r^{-2 \alpha}}{(\log r)^{2 \beta}}\right)$ we get 
    $$I_2 \leq C t^2 \int_0^{\frac{1}{t}} r\left( \frac{r^{-2 \alpha}}{(\log r)^{2 \beta}}\right) \leq  C \frac{t^{2\alpha}}{(\log \frac{1}{t})^{2\beta}}, $$ where $C$ is a positive constant. Hence, by combning $I_1$ and $I_2$ we get 
    $$\|M_tf-f\|_{L^2(S)}^2= \left( \frac{t^{2\alpha}}{(\log \frac{1}{t})^{2\beta}}\right)\,\,\,\text{as}\, t \rightarrow 0$$ proving \eqref{Dinilip}.
	\end{proof}
	This following Theorem can be proved in a similar way as Theorem \ref{Them4.1} by using the following version of the Hausdorff-Young inequality \cite[Theorem 4.6]{RS}
	\begin{equation} \label{HYtem}
	    \left( \int_{\mathbb{R}} \int_{N} |\widetilde{f}(\lambda+\gamma_{p'} \rho, n)|^{p'}\, dn |c(\lambda)|^{-2}\, d\lambda \right)^{\frac{1}{p'}} \leq C_p \|f\|_p
	\end{equation} and the fact that $\widetilde{M_t f}(z, n)= \widetilde{f}(z, n)\, \phi_z(a_t)$ for $t \in \mathbb{R}_+$ and $z \in S_p.$

\begin{thm}Let $S$ be a harmonic $NA$ group of dimension $d$ and let $\alpha, \beta >0.$ Let $f \in L^p(S), 1 < p \leq 2$ with $\frac{1}{p}+\frac{1}{p'}=1,$ be such that 
	\begin{equation} \label{vis1}
	    \|M_t f-f\|_{L^p(S)}= O \left( \frac{t^\alpha}{(\log \frac{1}{t})^\beta}\right)\,\,\text{as}\quad t \rightarrow 0. 
	    \end{equation}
	    Then 
	    \begin{equation} \label{Vish30}
	        \int_r^{\infty }\int_N |\widetilde{f}(\lambda+i\gamma_{p'} \rho, n)|^{p'} dn |c(\lambda)|^{-2} \,d\lambda = O(r^{-p' \alpha} (\log r)^{-p' \beta}) \,\,\,\,\text{as}\,\, r \rightarrow \infty.
	    \end{equation}
	\end{thm}
	\begin{proof}
	Suppose that \eqref{vis1} holds. If $\lambda \in \left[\frac{1}{t}, \frac{2}{t} \right], (\lambda, t \in \mathbb{R}_+)$ then $\lambda t \geq 1$ and so from Lemma \ref{estijac} we have 
	$$ \frac{1}{c^{p'}}|1-\phi_\lambda(a_t)|^{p'} \geq 1.$$ 
	Set $G(\lambda)= \int_{N} |\widetilde{f}(\lambda+i\gamma_{p'} \rho, n)|^{p'}\, dn.$  Now by using \eqref{HYtem} we get 
\begin{align*}
    \int_{\frac{1}{t}}^{\frac{2}{t}} G(\lambda) |c(\lambda)|^{-2}\, d\lambda & \leq  \frac{1}{c^{p'}} \int_{\frac{1}{t}}^{\frac{2}{t}}|1-\phi_\lambda(a_t)|^{p'} G(\lambda) |c(\lambda)|^{-2}\, d\lambda \\&\leq \frac{1}{c^{p'}} \int_{\frac{1}{t}}^{\frac{2}{t}} \int_{N} |1-\phi_\lambda(a_t)|^{p'}  |\widetilde{f}(\lambda+i\gamma_{p'} \rho, \,n)|^{p'}\, dn |c(\lambda)|^{-2}\, d\lambda \\&\leq \frac{1}{c^{p'}} \int_{\mathbb{R}} \int_{N} |1-\phi_\lambda(a_t)|^{p'}  |\widetilde{f}(\lambda+i\gamma_{p'} \rho, \,n)|^{p'}\, dn |c(\lambda)|^{-2}\, d\lambda \\&\leq \frac{C_p}{c^{p'}} \|\widetilde{M_t f- f}\|_{p'}^{p'}  \\& \leq C \|M_t f-f\|_{L^p(S)}^{p'}= O\left( \frac{t^{\alpha p'}}{(\log \frac{1}{t})^{\beta p'}} \right).
\end{align*}
Therefore, by putting $r=\frac{1}{t},$ we have 
$$\int_{r}^{2r} G(\lambda)\, |c(\lambda)|^{-2}\, d\lambda \leq C \left(\frac{r^{-\alpha p'}}{(\log r)^{\beta p'}}  \right).$$
Now, as in Theorem \ref{Them4.1} we get 
$$\int_r^{\infty} G(\lambda)\, |c(\lambda)|^{-2}\, d\lambda \leq C_\alpha\, \frac{r^{-\alpha p'}}{(\log r)^{\beta p'}}\,\,\,\,\,\text{as}\,\, r \rightarrow \infty,$$
with $C_\alpha= C (1-2^{p' \alpha})^{-1}.$
Consequently, we have 
	\begin{equation*}
	        \int_r^{\infty }\int_N |\widetilde{f}(\lambda, n)|^{p'} dn |c(\lambda)|^{-2} \,d\lambda = O(r^{-p' \alpha} (\log r)^{-p' \beta}) \,\,\,\,\text{as}\,\, r \rightarrow \infty,
	    \end{equation*} proving \eqref{Vish30}. \end{proof}
	
	The following theorem is an analogue of Theorem \ref{Titch2} for Dini-Lipschitz functions on harmonic $NA$ groups. This theorem was proved by Younis \cite{Younis} on circle and extended in \cite{DDR} to the setting of general compact homogeneoups spaces. This theorem is new even in the case of rank one noncompact symmetric spaces.
	
	\begin{thm} \label{DTitch2} Let $S$ be a harmonic NA group of dimension $d$. Let $0 <\alpha \leq 1,\,\gamma>0,\, 1<p \leq 2,$ and let $p'$ be  such that $\frac{1}{p}+\frac{1}{p'}=1.$
    Let $f \in L^p(S)$ be such that 
	\begin{equation} \label{vish12}
	    \|M_t f-f\|_{L^p(S)}= O \left( \frac{t^\alpha}{(\log \frac{1}{t})^\gamma}\right)\,\,\text{as}\,\, t \rightarrow 0. 
	    \end{equation} Then
    the function $F$ defined by 
$$F(\lambda):=   \left(\int_N |\widetilde{f}(\lambda+i \gamma_{p'} \rho, n)|^{p'} \,dn \right)^{\frac{1}{p'}},
$$
    belongs to $ L^\beta((0, \infty), |c(\lambda)|^{-2}\, d \lambda)$ provided that 
    \begin{align} \label{beta2}
        \frac{dp}{dp+\alpha p-d} < \beta \leq p'.
    \end{align}
	
	\end{thm}
	
	\begin{proof}
Arguing as in the proof of Theorem \ref{Titch2}, under the assumption of Theorem \ref{DTitch2} instead of \eqref{t9}  we get 
\begin{align}
    G (\Lambda) = O(\Lambda^{2\beta+d-\alpha \beta-\frac{d\beta}{p'}} (\log \Lambda)^{-\beta\gamma}).
    \end{align} 
    As a consequence,  we have the following $$\int_{1}^\Lambda F(\lambda)^\beta\, |c(\lambda)|^{-2}\, d\lambda =\Lambda^{-2\beta} G(\Lambda)+2 \beta \int_1^\Lambda \lambda^{-2\beta-1} G(\lambda)d\lambda = I_1+I_2.$$
 The term    $I_1$ is bounded under the condition on $\beta$ by using the above estimate of $G.$ Indeed 
 $$I_1= \Lambda^{-2 \beta} G(\Lambda) = O( \Lambda^{-\alpha \beta+d-\frac{d\beta}{p'}}(\log \Lambda)^{-\beta \gamma}).$$ The right hand side of the last estimate is bounded as $\Lambda \rightarrow \infty$ as we  must have $-\alpha \beta+d-\frac{d\beta}{p'}= \alpha \beta+d-d \beta (1-\frac{1}{p})= \frac{dp-\beta(dp+\alpha p-d)}{p}<0$ which is always the case by using the condition of $\beta,$ i.e., $\beta> \frac{dp}{dp+\alpha p-d}$ which also holds according to our assumption.  
 
 Similarly for $I_2,$ we get, using the estimate of $G(\lambda),$ 
 \begin{align}
     I_2= \int_1^{\Lambda} \lambda^{-2\beta-1} G(\lambda)\, d\lambda = O( \Lambda^{-\alpha \beta +d-\frac{d\beta}{p'}}(\log \Lambda)^{-\beta\gamma}). 
 \end{align}
 The right hand side of the above estimate is bounded because $-\alpha\beta+d-\frac{d \beta}{p'}= \frac{dp-\beta(d p+\alpha p-d)}{p} <0$ by using the assumption that $\beta> \frac{dp}{dp+\alpha p-d},$ which is equivalent to $dp-\beta (dp+\alpha p-d)<0.$ 
 
 Therefore, $$\int_1^{\Lambda} F(\lambda)^\beta |c(\lambda)|^{-2} d\lambda =O(1)$$ provided that $d-\alpha \beta -d \beta +\frac{d \beta}{p} <0$. Hence $F \in L^\beta((0, \infty), |c(\lambda)|^{-2} d\lambda).$ The proof is complete as conditions on $\beta$ are equivalent to \eqref{beta2}.
    \end{proof}

	\section{$L^p-L^q$ boundedness of multipliers for Jacobi transform for $1<p\leq 2\leq q<\infty$ } \label{Jacobinota}
	The purpose of this section is to extend a classical result of H\"ormander \cite{Hormander} on $L^p$-$L^q$ boundedness of Fourier multipliers with the use of inequalities of Paley type \cite{HP-1, HP} and Hausdorff-Young-Paley type for the Jacobi transform \cite{FK, Koorn}. We have seen earlier in Section \ref{Ess}, initially noted in Anker et al. \cite{Anker96}, that the spherical (radial) analysis on harmonic $NA$ groups perfectly fits into the  Jacobi analysis setting  studied by Koornwinder and his coauthors \cite{Koorn, FK, FK2}. So, instead of working in the setting of radial analysis on harmonic $NA$ we choose to work in the broad setting of Jacobi analysis. In  this section we  prove the Hausdorff-Young-Paley inequality and apply it to prove $L^p-L^q$ boundedness of Fourier multipliers in this setting. We first briefly describe the harmonic analysis related to the Jacobi differential operators can be found in papers \cite{Koorn, FK, FK2}.
	
	The hypergeomertic function $_2F_1(a, b; c, z)$ is defined by 
	$$_2F_1(a, b; c, z)= \sum_{k=0}^\infty \frac{(a)_k (b)_k}{(c)_k k!}z^k,\,\,\,\,\,\,|z|<1,$$ where $(a)_k$ is the Pochhammer symbol of $a$ given by $(a)_0=1$ and $(a)_k= a(a+1)\ldots(a+k-1).$ The function $z \mapsto\,\, _2F_1(a, b; c, z)$ is the unique solution of the differential equation
	$$z(1-z)u''(z)+(c-(a+b+1)z)u'(z)-ab u(z)=0,$$ which is regular at $0$ and equals to $1$ there. The Jacobi function with parameters $(\alpha, \beta)$ and $\lambda \in \mathbb{C}$ is defined by $\phi_\lambda^{\alpha, \beta}(t)= _2F_1(\frac{\alpha+\beta+1-i\lambda}{2}, \frac{\alpha+\beta+1+i\lambda}{2}; \alpha+1, -\sinh^2{t} ).$ It is well-known that the family $\{\phi_\lambda^{\alpha, \beta}\}_{\lambda \geq 0}$ forms an orthonormal system in $\mathbb{R}_+$ with respect to the weight $A_{\alpha, \beta}(t)= (2 \sinh{t})^{2\alpha+1} (2 \cosh{t})^{2\beta+1},\,\,t>0$ if $|\beta|<\alpha+1.$ In this section we assume that $\alpha \neq -1, -2, \ldots$ and $\alpha \geq \beta > \frac{-1}{2}.$ At times, for convenience we omit superscript and subscript $(\alpha, \beta)$ from $\phi_\lambda^{\alpha, \beta}$  and $A_{\alpha, \beta}(t)$ respectively and prefer to write them as $\phi_\lambda$ and $A(t)$ alone. It is worth noting that the following behavior for the weight $A_{\alpha, \beta}(t)$ holds true: 
	$$ A_{\alpha, \beta}(t) \asymp  \begin{cases} t^{2\alpha+1}\,\,\,& 0<t<1, \\ e^{2 \rho t} \,\,\,& t \geq 1.  \end{cases}$$
	Also, we denote $\alpha+\beta+1$ by $\rho.$ The Jacobi Laplacian is given by 
	$$\mathcal{L}_{\alpha, \beta}= \frac{d^2}{dt^2}+ ((2\alpha+1)\coth{t}+(2\beta+1)\tanh{t}) \frac{d}{dt},$$ where $\tanh{t}:=\frac{e^t-e^{-t}}{e^t+e^{-t}}$ and $\coth{t}:=\frac{1}{\tanh{t}}.$ The Jacobi functions can be alternatively characterized as the unique solutions of $$\mathcal{L}_{\alpha, \beta} \phi_\lambda^{\alpha, \beta}+(\lambda^2+\rho^2) \phi_\lambda^{\alpha, \beta}=0$$ on $\mathbb{R}_+$ such that $\phi_\lambda^{\alpha, \beta}(0)=1$ and $(\phi_\lambda^{\alpha, \beta})'(0)=0.$ Therefore, the function $t \mapsto \phi_\lambda^{\alpha, \beta}(t)$ is analytic for $t \geq 0.$ The function $\phi_\lambda^{\alpha, \beta}$ can be  estimated as 
	$$| \phi_\lambda^{\alpha, \beta}(t)| \leq \begin{cases}  1\, & \textnormal{if}\,\,\,\, |\textnormal{Im}(\lambda)| \leq \rho \\ e^{(\textnormal{Im}(\lambda)-\rho)t} & \textnormal{if}\,\,\, |\textnormal{Im}(\lambda)| > \rho 
	\end{cases}$$
	for $t \in \mathbb{R}_+.$ It is also known that $\phi_\lambda^{\alpha, \beta}$ is bounded if and only if $|\textnormal{Im}(\lambda)| \leq \rho.$
	
	Denote the space of all even smooth functions on $\mathbb{R}$ with compact support  by $C_c^\infty(\mathbb{R})^\#$ and the space consisting of restriction of functions in $C_c^\infty(\mathbb{R})^\#$ to $\mathbb{R}_+$ by $C_c^\infty(\mathbb{R}_+)^\#.$
	For a  function $f \in C_c^\infty(\mathbb{R})^\#,$  the Jacobi transform $\widehat{f}(\lambda),$ $\lambda \in \mathbb{C},$ is defined by 
	$$\widehat{f}(\lambda):= \int_{0}^\infty f(t) \phi_\lambda(t) A(t) dt.$$
	
	The following inversion formula holds for functions in $C_c^\infty(\mathbb{R})^\#:$
	$$f(t)= \frac{1}{2\pi} \int_0^\infty \widehat{f}(\lambda)\,\phi_\lambda(t) |c(\lambda)|^{-2}d\lambda,\,\,\,\,\,t \in \mathbb{R},  $$ where $c(\lambda)$ is a multiple of the  meromorphic Harish-Chandra function given by the formula
	$$c(\lambda)=\frac{2^{\rho-i\lambda} \Gamma(\alpha+1) \Gamma (i\lambda)}{ \Gamma(\frac{1}{2}(\rho+i\lambda)) \Gamma(\frac{1}{2}(\rho+i \lambda)-\beta)}.$$
	If we denote the measure $A(t)dt$ by $d\mu(t)$ then the space $L^p(\mathbb{R}_+, d\mu)$, $1\leq p<\infty,$ associated with the measure $d\mu,$ at times,  denoted by $L^p(\mu),$ is defined as the space of all even functions on $\mathbb{R}$ such that
	$$\|f\|_{L^p(\mu)}:=\left( \int_0^\infty |f(t)|^p\, A(t)dt \right)^{\frac{1}{p}}=\left( \int_0^\infty |f(t)|^p\, d\mu(t) \right)^{\frac{1}{p}}<\infty,\,\,\,\,1\leq p <\infty.$$
	It is well-known that the Jacobi transform $f \mapsto \widehat{f}$ extends to an isometry from $L^2(\mu)$ onto $L^2(\mathbb{R}_+, \frac{1}{2\pi} |c(\lambda)|^{-2}d\lambda).$ 
	
	Let $S_p:=\{\lambda \in \mathbb{C}: |\textnormal{Im}(\lambda)|<(1-\frac{2}{p}) \rho \}.$ Then, for $f \in L^p(\mu),$ $1\leq p <2$ with $\frac{1}{p}+\frac{1}{p'}=1,$ the Jacobi transform $\widehat{f}(\lambda)$ is well-defined and holomorphic in $S_p$ and for all $\lambda \in S_p,$ we have
	$$|\widehat{f}(\lambda)| \leq \|f\|_{L^p(\mu)} \|\phi_\lambda\|_{p}.$$ In particular, if $p=1,$ $\widehat{f}(\lambda)$ is continuous on $S_1$ and for all $\lambda \in S_1,$ we have $$|\widehat{f}(\lambda)| \leq \|f\|_{L^1(\mu)}.$$ 
	
	For our convenience we denote the measure $\frac{1}{2\pi} |c(\lambda)|^{-2}d\lambda$ by $d\kappa(\lambda).$ 
	
	The following theorem is the Hausdorff-Young inequality for Jacobi transform (see \cite{Sinayoung}).
	\begin{thm} \label{HYJacobi}
	Let $1 \leq p \leq 2$ with $\frac{1}{p}+\frac{1}{p'}=1.$ If $f\in L^p(\mathbb{R}_+, d\mu)$ then we have 
	$$\|\widehat{f}\|_{L^{p'}(\mathbb{R}_+, d\kappa(\lambda))} \leq \|f\|_{L^p(\mu)}.$$
	\end{thm}

	\subsection{Paley-type inequality for Jacobi transform} The following inequality can be regarded as a Paley-type inequality which  plays an important in proving the Hausdorff-Young-Paley inequality.

	\begin{thm}\label{Paley}
	      Suppose that $\psi$ is a positive function  on $\mathbb{R}_+$  satisfying the condition 
	    \begin{equation}
	        M_\psi := \sup_{t>0} t \int_{\underset{\psi(\lambda)>t}{\lambda \in \mathbb{R}_+}}  |c(\lambda)|^{-2}d\lambda <\infty.
	        \end{equation}
	        Then for  $f \in L^p(\mu),$ $1<p\leq 2,$ we have 
	        \begin{align} \label{Paleyin}
	            \left( \int_0^\infty |\widehat{f}(\lambda)|^p\, \psi(\lambda)^{2-p}  |c(\lambda)|^{-2}d\lambda \right)^{\frac{1}{p}} \lesssim M_{\psi}^{\frac{2-p}{p}}\, \|f\|_{L^p(\mu)}.
	        \end{align}
	\end{thm}
	\begin{proof} We follow the proof in \cite{ARN, AR}. Let us consider a measure $\nu$ on $\mathbb{R}_+$ given by \begin{equation} \label{nu}
	    \nu(\lambda)= \psi(\lambda)^2 d\kappa(\lambda),
	\end{equation}
	where $d\kappa=\frac{1}{2\pi} |c(\lambda)|^{-2}d\lambda.$
	We define the corresponding $L^p(\mathbb{R}_+, \nu)$-space, $1 \leq p<\infty,$ as the space of all complex-valued function $f$ defined by $\mathbb{R}_+$ such that 
	$$ \|f\|_{L^p(\mathbb{R}_+, \nu)}:= \left( \int_{\mathbb{R}_+} |f(\lambda)|^p \, \psi(\lambda)^2 d\kappa(\lambda) \right)^{\frac{1}{p}}<\infty.$$
	 We define a sublinear operator $T$ for $f \in L^p(\mu)$ by 
	$$ Tf(\lambda)= \frac{|\widehat{f}(\lambda)|}{\psi(\lambda)} \in L^p(\mathbb{R}_+, \nu).$$
We will show that $T$ is well-defined and bounded from $L^p(\mu)$ to $L^p(\mathbb{R}_+, \nu)$ for any $1 < p \leq 2.$	
In other words, we claim the following estimate:
\begin{equation} \label{vis10paley}
    \|Tf\|_{L^p(\mathbb{R}_+, \nu)}= \left( \int_{\mathbb{R}_+} \frac{|\widehat{f}(\lambda)|^p}{\psi(\lambda)^p}\,\psi(\lambda)^2 d\kappa(\lambda) \right)^{\frac{1}{p}} \lesssim M_{\psi}^{\frac{2-p}{p}} \|f\|_{L^p(\mu)},
\end{equation}
which will give us the required inequality \eqref{Paleyin} with $M_\psi := \sup_{t>0} \int_{\underset{\psi(\lambda)>t}{\lambda \in \mathbb{R}_+}} d\kappa(\lambda).$ 
	We will show that $T$ is weak-type $(2,2)$ and weak-type $(1,1).$  More precisely, with the distribution function,  
$$\nu(y; Tf)= \int_{\underset{\frac{|\widehat{f}(\lambda)|} {\psi(\lambda)} \geq y}{\lambda \in \mathbb{R}_+}} \psi(\lambda)^2 d\kappa(\lambda), $$ where $\nu$ is give by formula \eqref{nu}, we show that 
\begin{equation} \label{vish5.4}
    \nu(y; Tf) \leq \left( \frac{M_2 \|f\|_{2}}{y} \right)^2 \,\,\,\,\,\,\text{with norm}\,\, M_2=1,
\end{equation}
\begin{equation} \label{vish5.5}
    \nu(y; Tf) \leq \frac{M_1 \|f\|_{1}}{y}\,\,\,\,\,\,\text{with norm}\,\, M_1=M_\psi.
\end{equation} 
	Then the estimate \eqref{vis10paley} follows from the Marcinkiewicz interpolation Theorem. Now, to show \eqref{vish5.4}, using Plancherel identity we get
	\begin{align*}
    y^2 \nu(y; Tf)&\leq \sup_{y>0}y^2 \nu(y; Tf)= \|Tf\|^2_{L^{2, \infty}(\mathbb{R}, \nu)}  \leq \|Tf\|^2_{L^2(\mathbb{R}_+, \nu)} \\&= \int_{\mathbb{R}_+} \left( \frac{|\widehat{f}(\lambda)|}{\psi(\lambda)} \right)^2 \psi(\lambda)^2 d\kappa(\lambda) \\&= \int_{\mathbb{R}_+} |\widehat{f}(\lambda)|^2\, d\kappa(\lambda)  = \|f\|_2^2.  \end{align*}
Thus, $T$ is type $(2,2)$ with norm $M_2 \leq 1.$ Further, we show that $T$ is of weak type $(1,1)$ with norm $M_1=M_\psi$; more precisely, we show that 
\begin{align} \label{11weak}
    \nu \left\{ \lambda \in \mathbb{R}_+: \frac{|\widehat{f}(\lambda)|}{\psi(\lambda)}>y \right\} \lesssim M_\psi \frac{\|f\|_1}{y}.
\end{align}
The left hand side is an integral $ \int \psi(\lambda) d\kappa(\lambda)$ taken over all those $\lambda \in \mathbb{R}_+$ for which $\frac{|\widehat{f}(\lambda)|}{\psi(\lambda)}>y.$
Since $|\widehat{f}(\lambda)| \leq \|f\|_1$ for all $\lambda \in \mathbb{R}_+$ we have
$$\left\{ \lambda \in \mathbb{R}_+: \frac{|\widehat{f}(\lambda)|}{\psi(\lambda)}>y \right\} \subset \left\{ \lambda \in \mathbb{R}_+: \frac{\|f\|_1}{\psi(\lambda)}>y \right\},$$ 
 for any $y>0$ and, therefore,
$$\nu \left\{ \lambda \in \mathbb{R}_+: \frac{|\widehat{f}(\lambda)|}{\psi(\lambda)}>y \right\} \leq \nu \left\{ \lambda \in \mathbb{R}_+: \frac{\|f\|_1}{\psi(\lambda)}>y \right\}.$$

Now by setting $w:=\frac{\|f\|_1}{y},$ we have 
\begin{align}
    \nu \left\{ \lambda \in \mathbb{R}_+: \frac{\|f\|_1}{\psi(\lambda)}>y \right\} \leq  \int_{\overset{\lambda \in \mathbb{R}_+}{\psi(\lambda) \leq w} } \psi(\lambda)^2\, d\kappa(\lambda).
\end{align}
Now we claim that 
\begin{align} \label{claim}
    \int_{\overset{\lambda \in \mathbb{R}_+}{\psi(\lambda) \leq w} } \psi(\lambda)^2\, d\kappa(\lambda) \lesssim M_\psi w.
\end{align}
Indeed, first we notice that 
\begin{align*}
    \int_{\overset{\lambda \in \mathbb{R}_+}{\psi(\lambda) \leq w} } \psi(\lambda)^2\, d\kappa(\lambda) = \int_{\overset{\lambda \in \mathbb{R}_+}{\psi(\lambda) \leq w} } \, d\kappa(\lambda) \int_0^{\psi(\lambda)^2} d\tau. 
\end{align*}
By interchanging the order of integration we get 
\begin{align*}
    \int_{\overset{\lambda \in \mathbb{R}_+}{\psi(\lambda) \leq w} } \,d\kappa(\lambda) \int_0^{\psi(\lambda)^2} d\tau = \int_{0}^{w^2} d\tau \int_{\underset{\tau^{\frac{1}{2}} \leq \psi(\lambda) \leq w}{\lambda \in \mathbb{R}_+}} d\kappa(\lambda).
\end{align*}
Further, by making substitution  $\tau= t^2,$ it gives 
\begin{align*}
    \int_{0}^{w^2} d\tau \int_{\underset{\tau^{\frac{1}{2}} \leq \psi(\lambda) \leq w}{\lambda \in \mathbb{R}_+}} d\kappa(\lambda) &= 2 \int_0^w t\, dt \int_{\underset{t \leq \psi(\lambda) \leq w}{\lambda \in \mathbb{R}_+}} d\kappa(\lambda) \\&\leq 2 \int_0^w t\, dt \int_{\underset{t \leq \psi(\lambda) }{\lambda \in \mathbb{R}_+}} d\kappa(\lambda).
\end{align*}
Since 
$$ t \int_{\underset{t \leq \psi(\lambda) }{\lambda \in \mathbb{R}_+}} d\kappa(\lambda) \leq \sup_{t>0} t \int_{\underset{t \leq \psi(\lambda) }{\lambda \in \mathbb{R}_+}} d\kappa(\lambda) = M_\psi $$ is finite by assumption $M_\psi<\infty,$ we have 
\begin{align*}
    2 \int_0^w t\, dt \int_{\underset{t \leq \psi(\lambda) }{\lambda \in \mathbb{R}_+}} d\kappa(\lambda) \lesssim M_\psi w.
\end{align*}
This establishes our claim \eqref{claim} and eventually proves \eqref{11weak}. So, we have proved \eqref{vish5.4} and \eqref{vish5.5}. Then  by using the Marcinkiewicz interpolation theorem with $p_1=1$ and $p_2=2$ and $\frac{1}{p}= 1-\theta+\frac{\theta}{2}$ we now obtain
$$\left( \int_{\mathbb{R}_+} \left(\frac{|\widehat{f}(\lambda)|}{\psi(\lambda)} \right)^p \psi(\lambda)^2\, d\kappa(\lambda) \right)^{\frac{1}{p}}= \|Tf\|_{L^p(\mathbb{R}_+,\, \nu)} \lesssim M_\psi^{\frac{2-p}{p}} \|f\|_{L^p(\mu)}.$$
This completes the proof of the theorem. \end{proof}

	\subsection{Hausdorff-Young-Paley inequality for Jacobi transform}
	The following theorm \cite{BL} is useful to to obtain one of our crucial result.

\begin{thm} \label{interpolation} Let $d\mu_0(x)= \omega_0(x) d\mu'(x),$ $d\mu_1(x)= \omega_1(x) d\mu'(x),$ and write $L^p(\omega)=L^p(\omega d\mu')$ for the weight $\omega.$ Suppose that $0<p_0, p_1< \infty.$ Then 
$$(L^{p_0}(\omega_0), L^{p_1}(\omega_1))_{\theta, p}=L^p(\omega),$$ where $0<\theta<1, \, \frac{1}{p}= \frac{1-\theta}{p_0}+\frac{\theta}{p_1}$ and $\omega= \omega_0^{\frac{p(1-\theta)}{p_0}} \omega_1^{\frac{p\theta}{p_1}}.$
\end{thm} 

The following corollary is immediate.

\begin{cor}\label{interpolationoperator} Let $d\mu_0(x)= \omega_0(x) d\mu'(x),$ $d\mu_1(x)= \omega_1(x) d\mu'(x).$ Suppose that $0<p_0, p_1< \infty.$  If a continuous linear operator $A$ admits bounded extensions, $A: L^p(Y,\mu)\rightarrow L^{p_0}(\omega_0) $ and $A: L^p(Y,\mu)\rightarrow L^{p_1}(\omega_1) ,$   then, we there exists a bounded extension $A: L^p(Y,\mu)\rightarrow L^{b}(\omega) $ of $A$, where  $0<\theta<1, \, \frac{1}{b}= \frac{1-\theta}{p_0}+\frac{\theta}{p_1}$ and 
 $\omega= \omega_0^{\frac{b(1-\theta)}{p_0}} \omega_1^{\frac{b\theta}{p_1}}.$
\end{cor} 

Using the above corollary we now present the Hausdorff-Young-Paley inequality. 
\begin{thm}[Hausdorff-Young-Paley inequality] \label{HYP} Let $1<p\leq 2,$ and let   $1<p \leq b \leq p' < \infty,$ where $p'= \frac{p}{p-1}.$ If $\psi(\lambda)$ is a positive function on $\mathbb{R}_+$ such that 
 \begin{equation}
	        M_\psi := \sup_{t>0} t \int_{\underset{\psi(\lambda)>t}{\lambda \in \mathbb{R}_+}} d\kappa(\lambda)
	        \end{equation}
is finite then for every $f \in L^p(\mu)$ 
 we have
\begin{equation} \label{Vish5.9}
    \left( \int_{\mathbb{R}_+}  \left( |\widehat{f}(\lambda)| \psi(\lambda)^{\frac{1}{b}-\frac{1}{p'}} \right)^b d\kappa(\lambda)  \right)^{\frac{1}{b}} \lesssim M_\varphi^{\frac{1}{b}-\frac{1}{p'}} \|f\|_{L^p(\mu)}.
\end{equation}
\end{thm}
This naturally reduced to Hausdorff-Young inequality when $b=p'$ and Paley inequality \eqref{Paleyin} when $b=p.$
\begin{proof}
From Theorem \ref{Paley}, the operator  defined by 
$$Af(\lambda)= \widehat{f}(\lambda),\,\,\,\,\lambda \in \mathbb{R}_+$$
is bounded from $L^p(\mu)$ to $L^{p}(\mathbb{R_+},\omega_0  d\mu'),$ where $d\mu'(\lambda)=d\kappa(\lambda)$ and $\omega_{0}(\lambda)=  \psi(\lambda)^{2-p}.$ From Theorem \ref{HYJacobi}, we deduce that $A:L^p(\mu) \rightarrow L^{p'}(\mathbb{R}_+, \omega_1 d\mu')$ with $d\mu'(\lambda)=d\kappa(\lambda)$ and   $\omega_1(\lambda)= 1$  admits a bounded extension. By using the real interpolation (Corollary \ref{interpolationoperator} above) we will prove that $A:L^p(\mu) \rightarrow L^{b}(\mathbb{R}_+, \omega d\mu'),$ $p\leq b\leq p',$ is bounded,
where the space $L^p(\mathbb{R}_+,\, \omega d\mu')$ is defined by the norm 
$$\|\sigma\|_{L^p(\mathbb{R},\, \omega d\mu')}:=\left( \int_{ \mathbb{R}_+} |\sigma(\lambda)|^p w(\lambda) \,d\mu'(\lambda) \right)^{\frac{1}{p}}= \left( \int_{ \mathbb{R}_+} |\sigma(\lambda)|^p w(\lambda) d\kappa(\lambda) \right)^{\frac{1}{p}}$$
 and $\omega(\lambda)$ is positive function over $\mathbb{R}_+$ to be determined. To compute $\omega,$ we can use Corollary \ref{interpolationoperator}, by fixing $\theta\in (0,1)$ such that $\frac{1}{b}=\frac{1-\theta}{p}+\frac{\theta}{p'}$. In this case $\theta=\frac{p-b}{b(p-2)},$ and 
 \begin{equation}
     \omega= \omega_0^{\frac{p(1-\theta)}{p_0}} \omega_1^{\frac{p\theta}{p_1}}= \psi(\lambda)^{1-\frac{b}{p'}}.     
 \end{equation}
 Thus we finish the proof.
\end{proof}
The following theorem is the main result of this section which is an analogue of H\"ormander multiplier theorem  \cite{Hormander1960} for Jacobi transform.

\begin{thm} \label{Jacobimult}  Let $1<p \leq 2 \leq q<\infty$. Suppose that $T$ is a Jacobi-Fourier multiplier with symbol $h,$ that is, $$\widehat{Tf}(\lambda)= h(\lambda) \widehat{f}(\lambda),\,\,\,\lambda \in \mathbb{R}_+ ,$$
 where $h$ is an bounded measurable even function on $\mathbb{R}.$  Then we have 
$$\|T\|_{L^p(\mu) \rightarrow L^q(\mu)}\lesssim \sup_{s>0} s \left[ \int_{\{ \lambda \in \mathbb{R}_+: |h(\lambda)|>s\}} d\kappa(\lambda) \right]^{\frac{1}{p}-\frac{1}{q}}.$$
   \end{thm}
\begin{proof}
 Let us first assume that $p \leq q',$ where $\frac{1}{q}+\frac{1}{q'}=1.$ Since $q' \leq 2,$ the Hausdorff-Young inequality gives that 
 \begin{align*}
     \|Tf\|_{L^q(\mu)} \leq \|\widehat{Tf}\|_{L^{q'}(\mu)} = \|h \widehat{f}\|_{L^{q'}(\mu)}.
 \end{align*}
 
 The case $q' \leq (p')'=p$ can be reduced to the case $p \leq q'$ as follows. Using the duality of $L^p$-spaces we have $\|T\|_{L^p(\mu) \rightarrow L^q(\mu)}= \|T^*\|_{L^{q'}(\kappa) \rightarrow L^{p'}(\kappa)}.$ The symbol of adjoint operator $T^*$  is equal to $\check{h},$ which equal to  $h$ and obviously we have $|\check{h}|= |h|$ (see Proposition 4.10 in \cite{Anker96}). 
 Now, we are in a position to apply Theorem \ref{HYP}. Set $\frac{1}{p}-\frac{1}{q}=\frac{1}{r}.$ Now, by applying  Theorem \ref{HYP} with $\psi= |h|^r$ with $b=q'$ we get 
 $$\|h \widehat{f}\|_{L^{q'}(\kappa)} \lesssim \left(  \sup_{s>0} s \int_{\underset{|h(\lambda)|^r > s}{\lambda \in \mathbb{R}_+}} d\kappa(\lambda)   \right)^{\frac{1}{r}} \|f\|_{L^p(\mu)} $$ for all $f \in L^p(\mu),$ in view of $\frac{1}{p}-\frac{1}{q}=\frac{1}{q'}-\frac{1}{p'}=\frac{1}{r.}$ Thus, for $1<p \leq 2 \leq q<\infty,$ we obtain 
 
 $$\|Tf\|_{L^q(\mu)} \lesssim \left(  \sup_{s>0} s \int_{\underset{|h(\lambda)|^r > s}{\lambda \in \mathbb{R}_+}} d\kappa(\lambda)   \right)^{\frac{1}{r}} \|f\|_{L^p(\mu)}.$$
 Further, the proof follows from the following inequality: 
 \begin{align*}
     \left(  \sup_{s>0} s \int_{\underset{|h(\lambda)|^r > s}{\lambda \in \mathbb{R}_+}} d\kappa(\lambda)  \right)^{\frac{1}{r}} &= \left(  \sup_{s>0} s \int_{\underset{|h(\lambda)| > s^{\frac{1}{r}}}{\lambda \in \mathbb{R}_+}} d\kappa(\lambda)    \right)^{\frac{1}{r}} \\&=  \left(  \sup_{s>0} s^r \int_{\underset{|h(\lambda)| > s} {\lambda \in \mathbb{R}_+}} d\kappa(\lambda)    \right)^{\frac{1}{r}} \\&= \sup_{s>0} s \left(   \int_{\underset{|h(\lambda)| > s} {\lambda \in \mathbb{R}_+}} d\kappa(\lambda)    \right)^{\frac{1}{r}},
 \end{align*} proving Theorem \ref{Jacobimult}.
\end{proof}

As an immediate consequence of above Theorem \ref{Jacobimult} we have the following multiplier theorem on harmonic $NA$ groups.  

\begin{cor}
 Let $S$ be a harmonic $NA$ group and let $1<p \leq 2 \leq q<\infty.$ Suppose that $h$ is a bounded measurable even function on $\mathbb{R}.$ Then the Fourier multiplier defined by $T_h f:= \mathcal{H}^{-1}(h \mathcal{H}f)$ for $f \in C_c^\infty(S)^\#$ is bounded from $L^p(S)^\#$ to $L^q(S)^\#$ provided that
 $$\sup_{s>0} s \int_{\{ \lambda \in \mathbb{R}_+: |h(\lambda)|^{\frac{1}{p}-\frac{1}{q}}>s\}} |c(\lambda)|^{-2}\, d\lambda <\infty.$$
 Moreover, $$\|T_h\|_{L^p(S)^\# \rightarrow L^q(S)^\#} \lesssim \sup_{s>0} s \left(  \int_{\{ \lambda \in \mathbb{R}_+: |h(\lambda)|>s\}} |c(\lambda)|^{-2}\, d\lambda  \right)^{\frac{1}{p}-\frac{1}{q}}.$$
\end{cor}

Now, we apply Theorem \ref{Jacobimult} to prove the $L^p$-$L^q$ boundedness of spectral multipliers for operator $L:=-\mathcal{L}_{\alpha, \beta}.$ If $\varphi \in L^\infty(\mathbb{R}_+, d\mu),$ the spectral multiplier $\varphi(L),$ defined by $\varphi$ coincides with the Jacobi Fourier multiplier $T_h$ with $h(\lambda)= \varphi(\lambda^2+\rho^2)$ for $\lambda \in \mathbb{R}_+.$  The $L^p$-boundedness of spectral multipliers  has been proved by several authors in many different setting, e.g., Bessel transform \cite{BCC}, Dunkl harmonic oscillator \cite{Wro}. The $L^p$-$L^q$ boundedness of spectral multipliers for compact Lie groups, Heisenberg groups, graded Lie groups have been proved by R. Akylzhanov and the second author \cite{AR}. M. Chatzakou and the second author recently studied $L^p$-$L^q$ boundedness of spectral multiplier for the anharmonic oscillator \cite{CK} (see also \cite{CKNR}).

\begin{thm} \label{specmul}
Let $1<p \leq 2 \leq q <\infty$ and let $\varphi$ be a monotonically  decreasing continuous function on $[\rho^2, \infty)$ such that $\lim_{u \rightarrow \infty}\varphi(u)=0.$ Then we have 
\begin{equation}
    \|\varphi(L)\|_{\textnormal{op}} \lesssim \sup_{u>\rho^2} \varphi(u)  \begin{cases} (u-\rho^2)^{\frac{3}{2}(\frac{1}{p}-\frac{1}{q})} & \quad \textnormal{if} \quad (u-\rho^2)^{\frac{1}{2}} \leq 1, \\ (u-\rho^2)^{(\alpha+1)(\frac{1}{p}-\frac{1}{q})} & \quad \textnormal{if} \quad (u-\rho^2)^{\frac{1}{2}}>1,  \end{cases}
\end{equation} where  $\|\cdot\|_{\textnormal{op}} $ denotes the operator norm from $L^p(\mu)$ to $L^q(\mu).$ 
\end{thm}
\begin{proof} Since $\varphi(L)$ is a Jacobi-Fourier multiplier with the symbol $\varphi(\lambda^2+\rho^2),$ as an application of Theorem \ref{Jacobimult}, we get
\begin{align*}
    \|\varphi(L)\|_{\textnormal{op}} & \lesssim \sup_{s>0} s \left[ \int_{ \{ \lambda \in \mathbb{R}_+ :\, \varphi(\lambda^2+\rho^2) \geq s\} } |c(\lambda)|^{-2}\,d\lambda \right]^{\frac{1}{p}-\frac{1}{q}} \\ & = \sup_{0<s<\varphi(\rho^2)} s \left[ \int_{ \{ \lambda \in \mathbb{R}_+ :\, \varphi(\lambda^2+\rho^2) \geq s\} } |c(\lambda)|^{-2}\,d\lambda \right]^{\frac{1}{p}-\frac{1}{q}}, 
\end{align*} since $\varphi \leq \varphi(\rho^2).$ Now, as $s \in (0, \varphi(\rho^2)]$ we can write $s= \varphi(u)$ for some $u \in [\rho^2, \infty)$ and, therefore, we have 
\begin{align*}
    \|\varphi(L)\|_{\textnormal{op}} \lesssim \sup_{\varphi(u) <\varphi(\rho^2)} \varphi(u)  \left[ \int_{ \{ \lambda \in \mathbb{R}_+ :\, \varphi(\lambda^2+\rho^2) \geq \varphi(u)\} } |c(\lambda)|^{-2}\,d\lambda \right]^{\frac{1}{p}-\frac{1}{q}}.
\end{align*} Since $\varphi$ is monotonically decreasing we get 
\begin{align*}
    \|\varphi(L)\|_{\textnormal{op}} \lesssim \sup_{u >\rho^2} \varphi(u) \left[ \int_{ \{ \lambda \in \mathbb{R}_+ :\, \lambda \leq (u-\rho^2)^{\frac{1}{2}}\} } |c(\lambda)|^{-2}\,d\lambda \right]^{\frac{1}{p}-\frac{1}{q}}.
\end{align*}
Now, we use the estimate of the $c$-function to get 
\begin{align*}
     \|\varphi(L)\|_{\textnormal{op}} &\lesssim  \sup_{u >\rho^2} \varphi(u) \begin{cases}  \left[ \int_0^{ (u-\rho^2)^{\frac{1}{2}} } \lambda^2 \,d\lambda \right]^{\frac{1}{p}-\frac{1}{q}} \quad &\textnormal{if} \quad   (u-\rho^2)^{\frac{1}{2}} \leq 1, \\ \left[ \int_0^1 \lambda^2 d\lambda +\int_1^{ (u-\rho^2)^{\frac{1}{2}} } \lambda^{2\alpha+1}\,d\lambda \right]^{\frac{1}{p}-\frac{1}{q}} \quad &\textnormal{if} \quad  (u-\rho^2)^{\frac{1}{2}}>1
     \end{cases} \\&=\sup_{u>\rho^2} \varphi(u)  \begin{cases} (u-\rho^2)^{\frac{3}{2}(\frac{1}{p}-\frac{1}{q})} & \quad \textnormal{if} \quad (u-\rho^2)^{\frac{1}{2}} \leq 1, \\ (u-\rho^2)^{(\alpha+1)(\frac{1}{p}-\frac{1}{q})} & \quad \textnormal{if} \quad (u-\rho^2)^{\frac{1}{2}}>1,  \end{cases}
\end{align*} proving Theorem \ref{specmul}.
\end{proof}
\section*{Acknowledgment}
The authors thank Jean-Philippe Anker for discussions. Vishvesh Kumar thanks William O. Bray for his suggestions. VK and MR are supported  by the FWO  Odysseus  1  grant  G.0H94.18N:  Analysis  and  Partial Differential Equations and by the Methusalem programme of the Ghent University Special Research Fund (BOF)
(Grant number 01M01021). MR is also supported  by EPSRC grant EP/R003025/2.

\bibliographystyle{amsplain}

\end{document}